\documentclass[oneside, 11pt]{amsart} 

\usepackage{amsfonts}
\usepackage{amssymb}
\usepackage{latexsym}
\usepackage{graphics}
\usepackage[2emode]{psfrag}
\usepackage{amsthm}
\usepackage{amsmath}
\usepackage[all]{xy}

\begin{document}  

\newcommand{\norm}[1]{\| #1 \|}
\def\N{\mathbb N}
\def\Z{\mathbb Z}
\def\Q{\mathbb Q}
\def\mod{\textit{\emph{~mod~}}}
\def\R{\mathcal R}
\def\S{\mathcal S}
\def\*  C{{*  \mathcal C}} 
\def\C{\mathcal C}
\def\D{\mathcal D}
\def\J{\mathcal J}
\def\M{\mathcal M}
\def\T{\mathcal T}          

\newcommand{\Hom}{{\rm Hom}}
\newcommand{\End}{{\rm End}}
\newcommand{\Ext}{{\rm Ext}}
\newcommand{\Mor}{{\rm Mor}\,}
\newcommand{\Aut}{{\rm Aut}\,}
\newcommand{\Hopf}{{\rm Hopf}\,}
\newcommand{\Ann}{{\rm Ann}\,}
\newcommand{\Ker}{{\rm Ker}\,}
\newcommand{\Reg}{{\rm Reg}\,}
\newcommand{\Coker}{{\rm Coker}\,}
\newcommand{\Img}{{\rm Im}\,}
\newcommand{\coim}{{\rm Coim}\,}
\newcommand{\Trace}{{\rm Trace}\,}
\newcommand{\Char}{{\rm Char}\,}
\newcommand{\Mod}{{\rm mod}}
\newcommand{\Spec}{{\rm Spec}\,}
\newcommand{\sgn}{{\rm sgn}\,}
\newcommand{\Id}{{\rm Id}\,}
\newcommand{\Com}{{\rm Com}\,}
\newcommand{\codim}{{\rm codim}}
\newcommand{\Mat}{{\rm Mat}}
\newcommand{\can}{{\rm can}}
\newcommand{\sign}{{\rm sign}}
\newcommand{\kar}{{\rm kar}}
\newcommand{\rad}{{\rm rad}}

\def\lan{\langle}
\def\ran{\rangle}
\def\ot{\otimes}

\def\id{{\small \textit{\emph{1}}}\!\!1}    
\def\To{{\multimap\!\to}}
\def\bigperp{{\LARGE\textrm{$\perp$}}} 
\newcommand{\QED}{\hspace{\stretch{1}}
\makebox[0mm][r]{$\Box$}\\}

\def\RR{{\mathbb R}}
\def\FF{{\mathbb F}}
\def\NN{{\mathbb N}}
\def\CC{{\mathbb C}}
\def\DD{{\mathbb D}}
\def\ZZ{{\mathbb Z}}
\def\QQ{{\mathbb Q}}
\def\HH{{\mathbb H}}
\def\units{{\mathbb G}_m}
\def\GG{{\mathbb G}}
\def\EE{{\mathbb E}}
\def\FF{{\mathbb F}}
\def\rightact{\hbox{$\leftharpoonup$}}
\def\leftact{\hbox{$\rightharpoonup$}}

\newcommand{\Aa}{\mathcal{A}}
\newcommand{\Bb}{\mathcal{B}}
\newcommand{\Cc}{\mathcal{C}}
\newcommand{\Dd}{\mathcal{D}}
\newcommand{\Ee}{\mathcal{E}}
\newcommand{\Ff}{\mathcal{F}}
\newcommand{\Hh}{\mathcal{H}}
\newcommand{\Ii}{\mathcal{I}}
\newcommand{\Mm}{\mathcal{M}}
\newcommand{\Pp}{\mathcal{P}}
\newcommand{\Rr}{\mathcal{R}}
\def\*  C{{}*  \hspace*  {-1pt}{\Cc}}

\def\text#1{{\rm {\rm #1}}}

\def\smashco{\mathrel>\joinrel\mathrel\triangleleft}
\def\cosmash{\mathrel\triangleright\joinrel\mathrel<}

\def\Nat{\dul{\rm Nat}}

\renewcommand{\subjclassname}{\textup{2000} Mathematics Subject
     Classification}

\newtheorem{prop}{Proposition}[section] 
\newtheorem{lemma}[prop]{Lemma}
\newtheorem{cor}[prop]{Corollary}
\newtheorem{theo}[prop]{Theorem}
                       
\theoremstyle{definition}
\newtheorem{Def}[prop]{Definition}
\newtheorem{ex}[prop]{Example}
\newtheorem{exs}[prop]{Examples}
\newtheorem{Not}[prop]{Notation}
\newtheorem{Ax}[prop]{Axiom}
\newtheorem{rems}[prop]{Remarks}
\newtheorem{rem}[prop]{Remark}
\newtheorem{op}[prop]{Open problem}
\newtheorem{conj}[prop]{Conjecture}

\def\smashco{\mathrel>\joinrel\mathrel\triangleleft}
\def\curlarrow{\mathrel\sim\joinrel\mathrel>}

\title{Commutative pseudo equality algebras}

\author[Lavinia Corina Ciungu]{Lavinia Corina Ciungu} 

\begin{abstract}
Pseudo equality algebras were initially introduced by Jenei and $\rm K\acute{o}r\acute{o}di$ as a possible algebraic semantic for fuzzy type theory, and they have been revised by Dvure\v censkij and Zahiri under the name of JK-algebras. 
In this paper we define and study the commutative pseudo equality algebras. 
We give a characterization of commutative pseudo equality algebras and we prove that an invariant 
pseudo equality algebra is commutative if and only if its corresponding pseudo BCK(pC)-meet-semilattice 
is commutative. 
Other results consist of proving that every commutative pseudo equality algebra is a distributive lattice and 
every finite invariant commutative pseudo equality algebra is a symmetric pseudo equality algebra. 
We also introduce and investigate the commutative deductive systems of pseudo equality algebras. 
As applications of these notions and results we define and study the measures and measure-morphisms 
on pseudo equality algebras, we prove new properties of state pseudo equality algebras, and we introduce 
and investigate the pseudo-valuations on pseudo equality algebras. \\

\textbf{AMS classification (2000):} 03G25, 06F05, 06F35
\end{abstract}

\maketitle

\section{Introduction}

\emph{Fuzzy type theory} (FTT) has been developed by V. $\rm Nov\acute{a}k$ (\cite{Nov1}) as a fyzzy logic of higher order, the fuzzy version of the classical type theory of the classical logic of higher order.
Other formal systems of FTT have also been described by  V. $\rm Nov\acute{a}k$, and
all these models are \emph{implication-based}, while the models of the classical type theory are \emph{equality based} having the identity (equality) as the principal connective. 
Since the first algebraic models for the set of truth values of FTT are residuated lattices, their basic operations are $\wedge$ (meet), $\vee$ (join), $\odot$ (multiplication) and $\rightarrow$ (residuum). In fuzzy logic the last operation is a semantic interpretation of the implication, while the logical equivalence is intepreted by the biresiduum $x\leftrightarrow y=(x\rightarrow y)\wedge (y\rightarrow x)$. 
Thus a basic connective has a semantic interpretation by a derived operation.
In order to overcome this discrepancy, we need a specific algebra of truth values for the fuzzy type theory. 
The first version of such an algebra has been introduced by V. $\rm Nov\acute{a}k$ (\cite{Nov2}) under 
the name of \emph{EQ-algebra} and a new concept of fuzzy type theory has been developed based on EQ-algebras (\cite{Nov9}). 
A fuzzy-equality based logic called \emph{EQ-logic} has also been introduced (\cite{Nov10}), 
while the EQ-logics with delta connective were defined and investigated in \cite{Dyba1}. \\
According to \cite{Nov9}, a non-commutative EQ-algebra is an algebra $(E, \wedge, \odot, \thicksim, 1)$ of the type $(2, 2, 2, 0)$ such that the following axioms are fulfilled for all $x, y, z, u\in E$: \\
$(E_1)$ $(E,\wedge,1)$ is a commutative idempotent monoid w.r.t $\le$ ($x\le y$ defined as $x\wedge y=x),$ \\
$(E_2)$ $(E,\odot,1)$ is a monoid such that the operation $\odot$ is isotone w.r.t. $\le,$  \\
$(E_3)$ $x\thicksim x=1,$ 
$\hspace*{10.7cm}$ (\emph{reflexivity}) \\
$(E_4)$ $((x\wedge y)\thicksim z)\odot (u\thicksim x)\le z\thicksim(u\wedge y),$ 
$\hspace*{6cm}$ (\emph{substitution}) \\
$(E_5)$ $(x\thicksim y)\odot (z\thicksim u)\le (x\thicksim z)\thicksim (y\thicksim u),$ 
$\hspace*{6cm}$ (\emph{congruence}) \\
$(E_6)$ $(x\wedge y\wedge z)\thicksim x\le (x\wedge y)\thicksim x,$ 
$\hspace*{5cm}$ (\emph{isotonicity of implication}) \\
$(E_7)$ $(x\wedge y)\thicksim x\le (x\wedge y\wedge z)\thicksim (x\wedge z),$ 
$\hspace*{4cm}$ (\emph{antitonicity of implication}) \\
$(E_8)$ $x\odot y\le x\thicksim y$. 
$\hspace*{9.8cm}$ (\emph{boundedness}) \\
An EQ-algebra is \emph{commutative} if $\odot$ is commutative. \\
The operation $\thicksim$ is a fuzzy equality and the implication $\rightarrow$ is defined by 
$x\rightarrow y=(x\wedge y)\thicksim x$, hence the tie between multiplication and residuation is weaker than in the 
case of residuated lattices. In this sense, EQ-algebras generalize the residuated lattices. \\
As S. Jenei mentioned in \cite{Jen2}, if the product operation in EQ-algebras is replaced by another binary operation 
smaller or equal than the original product we still obtain an EQ-algebra, and this fact might make it difficult to 
obtain certain algebraic results. \\
For this reason, S. Jenei introduced in \cite{Jen2} a new structure, called \emph{equality algebra} consisting of two 
binary operations - meet and equivalence, and constant $1$. 
It was proved in \cite{Jen1}, \cite{Ciu1} that any equality algebra has a corresponding BCK-meet-semilattice 
satisying the \emph{contraction} condition (BCK(C)-meet-semilattice, for short)  
and any BCK(C)-meet-semilattice has a corresponding equality algebra. 
Since the equality algebras could also be ``candidates" for a possible algebraic semantics for fuzzy type theory, 
their study is highly motivated. 
As a generalization of equality algebras, Jenei and $\rm K\acute{o}r\acute{o}di$ introduced in \cite{Jen1} 
a concept of \emph{pseudo equality algebras} and proved that the pseudo equality algebras are term equivalent to 
pseudo BCK-meet-semilattices. 
In \cite{Ciu1} a gap was found in the proof of this result and a counterexample was given as well as a correct 
version of it. Moreover, Dvure\v censkij and Zahiri showed in \cite{Dvu7} that every pseudo equality algebra in the sense of \cite{Jen1} is an equality algebra and they defined and investigated a new concept of pseudo equality algebras 
(called JK-algebras) and established a connection between pseudo equality algebras and a special class of pseudo BCK-meet-semilattices (pseudo BCK(pC)-meet-semilattices).  
The internal states on a pseudo equality algebra have been introduced and investigated in \cite{Ciu5}, while state 
pseudo equality algebras were studied in \cite{Ciu8}. 
Apart from their logical interest, equality algebras as well as pseudo equality algebras seem to have important algebraic properties and it is worth studying them from an algebraic point of view. 
Commutative BCK-algebras were studies in \cite{Dvu2, DvPu, Dvu3}, while commutative pseudo BCK-algebras 
were originally defined by G. Georgescu and A. Iorgulescu in \cite{Geo15} 
under the name of \emph{semilattice-ordered pseudo BCK-algebras}. Properties of these structures were 
investigated by J. K{\"u}hr in \cite{Kuhr6, Kuhr2}. 
A characterization of commutative pseudo BCK-algebras is given in \cite{Ciu7}, where the commutative 
deductive systems of pseudo BCK-algebras are defined and investigated. \\
States or measures give a probabilistic interpretation of randomness of events of given algebraic structures. 
For MV-algebras, Mundici introduced states (an analogue of probability measures) in 1995, \cite{Mun1}, as 
averaging of the truth-value in \L ukasiewicz logic.
Measures on BCK algebras were introduced and studied by A. Dvure\v censkij in \cite{Dvu1, DvPu, Dvu4}.
Measures on pseudo BCK-algebras were introduced and studied in \cite{Ciu6}, and it was proved that 
the quotient pseudo BCK-algebra that is downwards-directed over the kernel of a measure can
be embedded as a subalgebra into the negative cone of an abelian and Archimedean $\ell$-group. 
Pseudo-valuations were introduced and studied for residuated lattices (\cite{Bus3}), BCK-algebras 
(\cite{DoKa1}, \cite{Jun1}), BE-algebras (\cite{Lee1}), while the notion of a commutative pseudo-valuation  
was defined in \cite{DoKa2} for BCK-algebras. \\
In this paper we define and study the commutative pseudo equality algebras. 
We give a characterization of commutative pseudo equality algebras and we prove that an invariant  
pseudo equality algebra is commutative if and only if its corresponding pseudo BCK(pC)-meet-semilattice 
is commutative. Other results consist of proving that every commutative pseudo equality algebra is a distributive lattice and every finite invariant commutative pseudo equality algebra is a symmetric pseudo equality algebra. 
We introduce the notion of a commutative deductive system of a pseudo equality algebra and we give equivalent  
conditions for this notion. 
It is proved that a normal deductive system $H$ of a pseudo equality algebra $A$ is commutative if and only if 
$A/H$ is a commutative pseudo equality algebra.
As applications of the above mentioned notions and results we define and study the measures and measure-morphisms 
on pseudo equality algebras, and we prove new properties of state pseudo equality algebras. 
We prove that any measure-morphism on a pseudo equality algebra is a measure on it, and the kernel of a measure 
is a commutative deductive system. 
We show that the quotient pseudo equality algebra over the kernel of a measure is a commutative pseudo equality 
algebra. 
It is also proved that a pseudo equality algebra possessing an order-determining system is commutative. 
Other main results consist of proving that the measures on a pseudo equality algebra and the and measure-morphisms on 
a linearly ordered pseudo equality algebra coincide with those on its corresponding pseudo BCK(pC)-meet semilattice.   
We prove that the two types of internal states on a pseudo equality algebra coincide if and 
only if it is a commutative pseudo equality algebra.
If moreover the pseudo equality algebra is symmetric and linearly ordered, then we show that the state-morphisms coincide with the two types of internal states. 
The notions of pseudo-valuation and commutative pseudo-valuation on pseudo equality algebras are defined and investigated. Given a pseudo equality algebra $A$, it is proved that the kernel of a commutative pseudo-valuation on $A$ is a commutative deductive system of $A$. If moreover $A$ is commutative, then we prove 
that any pseudo-valuation on $A$ is commutative.

$\vspace*{5mm}$

\section{Preliminaries on pseudo equality algebras}

Pseudo equality algebras have been firstly defined by Jenei and $\rm K\acute{o}r\acute{o}di$ in \cite{Jen1} 
as a generalization of equality algebras. 
Dvure\v censkij and Zahiri showed in \cite{Dvu7} that every pseudo equality algebra in the sense of 
\cite{Jen1} is an equality algebra. They also defined and investigated a new concept of pseudo equality algebras 
(JK-algebras) and established a connection between pseudo equality algebras and a special class of pseudo BCK-meet-semilattices.
In this section we recall the main notions and results and we present new properties of pseudo equality algebras. 

\begin{Def} \label{ps-eq-05} $\rm($\cite{Dvu7}$\rm)$
A \emph{pseudo equality algebra} (or a \emph{JK-algebra}) is an algebra 
$\mathcal{A}=(A, \wedge, \thicksim, \backsim, 1)$ of the type $(2, 2, 2, 0)$ such that the following axioms are fulfilled for all $x, y, z\in A$: \\
$(A_1)$ $(A, \wedge, 1)$ is a meet-semilattice with top element $1,$ \\
$(A_2)$ $x \thicksim x = x \backsim x = 1,$ \\ 
$(A_3)$ $x \thicksim 1 = 1 \backsim x = x,$ \\
$(A_4)$ $x \le y \le z$ implies  
$x \thicksim z \le y \thicksim z$, $x \thicksim z \le x \thicksim y$, $z \backsim x \le z \backsim y$ and  
$z \backsim x \le y \backsim x,$ \\ 
$(A_5)$ $x \thicksim y \le (x \wedge z) \thicksim (y \wedge z)$  and 
        $x \backsim y \le (x \wedge z) \backsim (y \wedge z),$ \\  
$(A_6)$ $x \thicksim y \le (z \thicksim x) \backsim (z \thicksim y)$ and 
        $x \backsim y \le (x \backsim z) \thicksim(y \backsim z),$ \\
$(A_7)$ $x \thicksim y \le (x \thicksim z) \thicksim (y \thicksim z)$ and 
        $x \backsim y \le (z \backsim x) \backsim (z \backsim y)$. 
\end{Def}        
        
The operation $\wedge$ is called \emph{meet}(\emph{infimum}) and $\thicksim$, $\backsim$ are called 
\emph{equality} operations. We write $x \le y$ (and $y \ge x$) iff $x \wedge y = x$. \\
In the algebra $\mathcal{A}$ other two operations are defined, called \emph{implications}: \\
$\hspace*{5cm}$ $x \rightarrow y = (x\wedge y) \thicksim x$ \\
$\hspace*{5cm}$ $x \rightsquigarrow y = x \backsim (x \wedge y)$. \\
In the sequel we will also refer to the pseudo equality algebra $(A, \wedge, \thicksim, \backsim, 1)$ 
by its universe $A$. 
We will agree that $\thicksim$, $\backsim$, $\rightarrow$ and $\rightsquigarrow$ have higher priority than the operation $\wedge$. \\ 
A pseudo equality algebra $A$ is called \emph{bounded} if there exists an element $0\in A$ such that $0\le x$ 
for all $x\in A$. A bounded pseudo equality algebra is denoted by $(A, \wedge, \thicksim, \backsim, 0, 1)$.  

\begin{prop} \label{ps-eq-05-10} $\rm($\cite{Dvu7}$\rm)$ 
In any pseudo equality algebra $(A, \wedge, \thicksim, \backsim, 1)$ the following hold for all $x, y, z\in A$: \\
$(1)$ $x\wedge y\thicksim x\le x\wedge y\wedge z\thicksim x\wedge z$ and 
      $x\rightarrow y\le x\wedge z\le x\wedge z\rightarrow y;$ \\ 
$(2)$ $x\backsim x\wedge y\le x\wedge z\backsim x\wedge y\wedge z$ and 
      $x\rightsquigarrow y\le x\wedge z\rightsquigarrow y$. 
\end{prop}

\begin{prop} \label{ps-eq-10} $\rm($\cite{Dvu7}$\rm)$ 
In any pseudo equality algebra $(A, \wedge, \thicksim, \backsim, 1)$ the following hold for all $x, y, z\in A$: \\
$(1)$ $x\thicksim y\le y\rightarrow x$ and $x\backsim y\le x\rightsquigarrow y;$ \\
$(2)$ $x\le ((y\thicksim x)\backsim y)\wedge (y \thicksim (x\backsim y));$ \\
$(3)$ $x\backsim y=1$ or $y\thicksim x=1$ imply $x\le y;$ \\
$(4)$ $x\thicksim y=1$ implies $z\thicksim x\le z\thicksim y$ and 
      $x\backsim y=1$ implies $y\backsim z\le x\backsim z;$ \\
$(5)$ $x\le y$ iff $x\rightarrow y=1$ iff $x\rightsquigarrow y=1;$ \\
$(6)$ $1\rightarrow x=1\rightsquigarrow x=x$, 
      $x\rightarrow 1=x\rightsquigarrow 1=x\rightarrow x=x\rightsquigarrow x=1$, 
      $x\rightarrow x=x\rightsquigarrow x=1;$ \\
$(7)$ $x\le (y \rightarrow x)\wedge (y \rightsquigarrow x);$ \\
$(8)$ $x\le ((x \rightarrow y)\rightsquigarrow y)\wedge ((x \rightsquigarrow y)\rightarrow y);$ \\
$(9)$ $x\rightarrow y\le (y\rightarrow z)\rightsquigarrow (x\rightarrow z)$ and 
      $x\rightsquigarrow y\le (y\rightsquigarrow z)\rightarrow (x\rightsquigarrow z);$ \\  
$(10)$ $x\le y\rightarrow z$ iff $y\le x\rightsquigarrow z;$ \\
$(11)$ $x\rightarrow (y\rightsquigarrow z)=y\rightsquigarrow (x\rightarrow z);$ \\
$(12)$ $x\rightarrow y\le (x\wedge z)\rightarrow (y\wedge z)$ and  
       $x\rightsquigarrow y\le (x\wedge z)\rightsquigarrow (y\wedge z);$ \\
$(13)$ $x\rightarrow y=x\rightarrow (x\wedge y)$ and $x\rightsquigarrow y=x\rightsquigarrow (x\wedge y);$ \\
$(14)$ $1\thicksim x=x\backsim 1;$ \\
$(15)$ if $x\le y$, then $x\le (x\thicksim y)\wedge (y\backsim x);$ \\
$(16)$ $x\thicksim y\le 1\thicksim (y\thicksim x)$ and $x\backsim y\le 1\thicksim (y\backsim x)$. 
\end{prop}

\begin{prop} \label{ps-eq-20} $\rm($\cite{Ciu8}$\rm)$
In any pseudo equality algebra $(A, \wedge, \thicksim, \backsim, 1)$ the following hold for all $x, y\in A$:\\
$(1)$ $y\le (x\wedge y\thicksim x)\wedge (x\backsim x\wedge y);$ \\ 
$(2)$ if $x\le y$, then $x\le (x\thicksim y)\wedge (y\backsim x);$ \\
$(3)$ $x\le ((x\wedge y\thicksim x)\backsim y)\wedge (y\thicksim (x\backsim x\wedge y));$ \\
$(4)$ $y\le ((x\wedge y\thicksim x)\backsim y)\wedge (y\thicksim (x\backsim x\wedge y));$ \\
$(5)$ $x\thicksim y\le x\wedge y\thicksim y$ and $x\backsim y\le x\backsim x\wedge y$. 
\end{prop}

\begin{prop} \label{ps-eq-20-10} $\rm($\cite{Ciu8}$\rm)$
Let $(A, \wedge, \thicksim, \backsim, 1)$ be a pseudo equality algebra and let $x, y\in A$ such that $x\le y$.  Then the following hold for all $z\in A$:\\
$(1)$ $y\wedge z\thicksim y\le x\wedge z\thicksim x$ and $y\backsim y\wedge z\le x\backsim x\wedge z;$ \\ 
$(2)$ $z\wedge x\thicksim z\le z\wedge y\thicksim z$ and $z\backsim z\wedge x\le z\backsim z\wedge y$.
\end{prop}

\begin{prop} \label{ps-eq-30} $\rm($\cite{Ciu8}$\rm)$
Let $(A, \wedge, \thicksim, \backsim, 1)$ be a pseudo equality algebra. Then the following hold for all $x, y\in A$: \\
$(1)$ $y\thicksim ((x\wedge y\thicksim x)\backsim y)=x\wedge y\thicksim x;$ \\
$(2)$ $(y\thicksim (x\backsim x\wedge y))\backsim y=x\backsim x\wedge y$. 
\end{prop}

\begin{prop} \label{ps-eq-30-10}
Let $(A, \wedge, \thicksim, \backsim, 1)$ be a pseudo equality algebra. Then the following hold for all $x, y\in A$: \\
$(1)$ $x\le (x\thicksim ((x\wedge y\thicksim x)\backsim y))\wedge (((x\wedge y\thicksim x)\backsim y)\backsim x);$ \\
$(2)$ $x\le (x\thicksim (y\thicksim (x\backsim x\wedge y)))\wedge ((y\thicksim (x\backsim x\wedge y))\backsim x);$ \\
$(3)$ $y\le (y\thicksim ((x\wedge y\thicksim x)\backsim y))\wedge (((x\wedge y\thicksim x)\backsim y)\backsim y);$ \\
$(4)$ $y\le (y\thicksim (y\thicksim (x\backsim x\wedge y)))\wedge ((y\thicksim (x\backsim x\wedge y))\backsim y)$. 
\end{prop}
\begin{proof}
$(1)$ By Proposition \ref{ps-eq-20}$(3)$ we have $x\le (x\wedge y\thicksim x)\backsim y$ and applying 
Proposition \ref{ps-eq-20}$(2)$ we get  
$x\le (x\thicksim ((x\wedge y\thicksim x)\backsim y))\wedge (((x\wedge y\thicksim x)\backsim y)\backsim x)$. \\
$(2)$ Similarly as $(1)$, from $x\le y\thicksim (x\backsim x\wedge y)$ and applying Proposition \ref{ps-eq-20}$(2)$. \\
$(3)$ By Proposition \ref{ps-eq-20}$(4)$ we have $y\le (x\wedge y\thicksim x)\backsim y$ and applying 
Proposition \ref{ps-eq-20}$(2)$ we get 
$y\le (y\thicksim ((x\wedge y\thicksim x)\backsim y))\wedge (((x\wedge y\thicksim x)\backsim y)\backsim y)$. \\
$(4)$ Similarly as $(3)$.
\end{proof}

Denote  $x\vee_1 y=(x\wedge y\thicksim x)\backsim y$ and $x\vee_2 y=y\thicksim (x\backsim x\wedge y)$. 

\begin{prop} \label{ps-eq-30-20} In any pseudo equality algebra $A$ the following hold for all $x, y, x_1, x_2\in A$: \\
$(1)$ $1\vee_1 x=x\vee_1 1=1\vee_2 x=x\vee_2 1;$ \\
$(2)$ $x\le y$ implies $x\vee_1 y=x\vee_2 y=y;$ \\
$(3)$ $x\vee_1 x=x\vee_2 x=x;$ \\
$(4)$ $x, y\le x\vee_1 y, x\vee_2 y;$ \\
$(5)$ if $x_1\le x_2$, then $x_1\vee_1 y\le x_2\vee_1 y$ and $x_1\vee_2 y\le x_2\vee_2 y;$ \\ 
$(6)$ $x\wedge y\thicksim x=y\thicksim x\vee_1 y$ and $x\backsim x\wedge y=x\vee_2 y\backsim y$.  
\end{prop} 
\begin{proof}
$(1)$, $(2)$, $(3)$ are straightforward. \\
$(4)$ It follows by Proposition \ref{ps-eq-20}$(3)$,$(4)$. \\
$(5)$ Applying Proposition \ref{ps-eq-20-10}$(1)$ for $x:=x_1$, $y:=x_2$, $z:=y$ we get 
      $x_2\wedge y\thicksim x_2\le x_1\wedge y\thicksim x_1$. \\
By $(A_4)$ for $y\le x_2\wedge y\thicksim x_2\le x_1\wedge y\thicksim x_1$, it folows that 
$(x_1\wedge y\thicksim x_1)\backsim y\le (x_2\wedge y\thicksim x_2)\backsim y$, that is $x_1\vee_1 y\le x_2\vee_1 y$. 
From Proposition \ref{ps-eq-20-10}$(1)$ for $x:=x_1$, $y:=x_2$, $z:=y$ we get 
      $x_2\backsim x_2\wedge y\le x_1\backsim x_1\wedge y$. 
By $(A_4)$ for $y\le x_2\backsim x_2\wedge y\le x_1\backsim x_1\wedge y$ we get 
$y\thicksim (x_1\backsim x_1\wedge y)\le y\thicksim (x_2\backsim x_2\wedge y)$, that is $x_1\vee_2 y\le x_2\vee_2 y$. \\
$(6)$ It follows by Proposition \ref{ps-eq-30}. 
\end{proof}

Pseudo BCK-algebras were introduced by G. Georgescu and A. Iorgulescu in \cite{Geo15} as algebras 
with "two differences", a left- and right-difference, instead of one $*$ and with a constant element $0$ as the least element. Nowadays pseudo BCK-algebras are used in a dual form, with two implications, $\to$ and $\rightsquigarrow$ and with one constant element $1$, that is the greatest element. Thus such pseudo BCK-algebras are in the "negative cone" and are also called "left-ones". \\
A \emph{pseudo BCK-algebra} (more precisely, \emph{reversed left-pseudo BCK-algebra}) is a structure 
${\mathcal B}=(B,\leq,\rightarrow,\rightsquigarrow,1)$ where $\leq$ is a binary relation on $B$, $\rightarrow$ and $\rightsquigarrow$ are binary operations on $B$ and $1$ is an element of $B$ satisfying, for 
all $x, y, z \in B$, the  axioms:\\
$(B_1)$ $x \rightarrow y \leq (y \rightarrow z) \rightsquigarrow (x \rightarrow z)$, $\:\:\:$ 
            $x \rightsquigarrow y \leq (y \rightsquigarrow z) \rightarrow (x \rightsquigarrow z);$ \\ 
$(B_2)$ $x \leq (x \rightarrow y) \rightsquigarrow y$,$\:\:\:$ $x \leq (x \rightsquigarrow y) \rightarrow y;$ \\
$(B_3)$ $x \leq x;$ \\
$(B_4)$ $x \leq 1;$ \\
$(B_5)$ if $x \leq y$ and $y \leq x$, then $x = y;$ \\
$(B_6)$ $x \leq y$ iff $x \rightarrow y = 1$ iff $x \rightsquigarrow y = 1$. \\

Since the partial order $\leq$ is determined by either of the two ``arrows", we can eliminate $\leq$ from the signature 
and denote a pseudo BCK-algebra by ${\mathcal B}=(X,\rightarrow,\rightsquigarrow,1)$. \\
An equivalent definition of a pseudo BCK-algebra is given in \cite{Kuhr5}. \\
The structure ${\mathcal B}=(B,\rightarrow,\rightsquigarrow,1)$ of the type $(2,2,0)$ is a pseudo BCK-algebra iff it satisfies the following identities and quasi-identity, for all $x, y, z \in B$:\\ 
$(B_1')$ $(x \rightarrow y) \rightsquigarrow [(y \rightarrow z) \rightsquigarrow (x \rightarrow z)]=1;$ \\
$(B_2')$ $(x \rightsquigarrow y) \rightarrow [(y \rightsquigarrow z) \rightarrow (x \rightsquigarrow z)]=1;$ \\
$(B_3')$ $1 \rightarrow x = x;$ \\
$(B_4')$ $1 \rightsquigarrow x = x;$ \\
$(B_5')$ $x \rightarrow 1 = 1;$ \\
$(B_6')$ $(x\rightarrow y =1$ and $y \rightarrow x=1)$ implies $x=y$. \\
The partial order $\le$ is defined by $x \le y$ iff $x \rightarrow y =1$ (iff $x \rightsquigarrow y =1$). \\ 
If the poset $(B, \le)$ is a meet-semilattice, then ${\mathcal B}$ is called a \emph{pseudo BCK-meet-semilattice} and 
we denote it by $\mathcal{B}=(B, \wedge, \rightarrow, \rightsquigarrow, 1)$.       
If $(B,\leq)$ is a lattice, then we will say that ${\mathcal B}$ is a \emph{pseudo BCK-lattice} and 
it is denoted by $\mathcal{B}=(B, \wedge, \vee, \rightarrow, \rightsquigarrow, 1)$. 

\begin{lemma} \label{ps-eq-40} $\rm($\cite{Geo15}, {Ior1}$\rm)$ In any pseudo BCK-algebra 
$(B,\rightarrow,\rightsquigarrow,1)$ the following hold for all $x, y, z\in B$: \\ 
$(1)$ $x\le y$ implies $z\rightarrow x \le z\rightarrow y$ and $z\rightsquigarrow x \le z\rightsquigarrow y;$ \\
$(2)$ $x\le y$ implies $y\rightarrow z \le x\rightarrow z$ and $y\rightsquigarrow z \le x\rightsquigarrow z;$ \\
$(3)$ $x\rightarrow y \le (z\rightarrow x)\rightarrow (z\rightarrow y)$ and 
      $x\rightsquigarrow y \le (z\rightsquigarrow x)\rightsquigarrow (z\rightsquigarrow y);$ \\
$(4)$ $x\rightarrow (y\rightsquigarrow z)=y\rightsquigarrow (x\rightarrow z)$ and 
      $x\rightsquigarrow (y\rightarrow z)=y\rightarrow (x\rightsquigarrow z);$ \\
$(5)$ $x\le y\rightarrow x$ and $x\le y\rightsquigarrow x$.       
\end{lemma}

For more details about the properties of a pseudo BCK-algebra we refer te reader to \cite{Ior14} and \cite{Ciu2}. \\
Let $B$ be a pseudo BCK-algebra. The subset $D \subseteq B$ is called a \emph{deductive system} of $B$ if 
it satisfies the following conditions:\\
$(i)$ $1 \in D;$ \\
$(ii)$ for all $x, y \in B$, if $x, x \rightarrow y \in D$, then $y \in D$. \\ 
Condition $(ii)$ is equivalent to the condition: \\
$(ii^{\prime})$ for all $x, y \in B$, if $x, x \rightsquigarrow y \in D$, then $y \in D$. \\ 

A deductive system $D$ of a pseudo BCK-algebra $B$ is said to be \emph{normal} if it satisfies the condition:\\
$(iii)$ for all $x, y \in B$, $x \rightarrow y \in D$ iff $x \rightsquigarrow y \in D$. \\
We will denote by ${\mathcal DS}_{BCK}(B)$ the set of all deductive systems and by ${\mathcal DS}_{n_{BCK}}(B)$ 
the set of all normal deductive systems of a pseudo BCK-algebra $B$. \\
Obviously $\{1\}, B \in {\mathcal DS}_{BCK}(B), {\mathcal DS}_{n_{BCK}}(B)$ and 
${\mathcal DS}_{n_{BCK}}(B)\subseteq {\mathcal DS}_{BCK}(B)$.
For every subset $X\subseteq B$, the smallest deductive system of $B$ containing $X$ (i.e. the intersection of all deductive systems $D\in{\mathcal DS}_{BCK}(B)$ such that $X\subseteq D$) is called the deductive 
system \emph{generated by} $X$ and it will be denoted by $[X)$. If $X=\{x\}$ we write $[x)$ instead of $[\{x\})$. \\
A pseudo BCK-meet-semilattice with the \emph{$\rm($pD$\rm)$ condition} (i.e. with the  \emph{pseudo-distributivity} condition) or a \emph{pseudo BCK$\rm($pD$\rm)$-meet-semilattice} for short, is a pseudo BCK-meet-semilattice 
$(X, \wedge, \rightarrow, \rightsquigarrow, 1)$ satisfying the (pD) condition:\\
(pD) $\hspace*{2.1cm}$   $x\rightarrow (y\wedge z)=(x\rightarrow y)\wedge (x\rightarrow z)$, \\
     $\hspace*{3cm}$ $x\rightsquigarrow (y\wedge z)=(x\rightsquigarrow y)\wedge (x\rightsquigarrow z)$ \\ 
for all $x, y, z\in X$. \\ 
A pseudo BCK-meet-semilattice with the \emph{$\rm($pC$\rm)$ condition} (i.e. with the  \emph{pseudo-contraction} condition) or a \emph{pseudo BCK$\rm($pC$\rm)$-meet-semilattice} for short, is a pseudo BCK-meet-semilattice 
$(X, \wedge, \rightarrow, \rightsquigarrow, 1)$ satisfying the (pC) condition:\\
(pC) $\hspace*{2.1cm}$   $x\rightarrow y\le (x\wedge z)\rightarrow (y\wedge z)$, \\
     $\hspace*{3cm}$ $x\rightsquigarrow y\le (x\wedge z)\rightsquigarrow (y\wedge z)$ \\ 
for all $x, y, z\in X$. 

\begin{prop} \label{ps-eq-50} $\rm($\cite{Ciu8}$\rm)$
Any pseudo BCK(pD)-meet-semilattice is a pseudo BCK(pC)-meet-semilattice. 
\end{prop}

If $(A, \wedge, \rightarrow, \rightsquigarrow, 1)$ is a 
pseudo BCK(pC)-meet-semilattice, then $x\rightarrow x\wedge y=x\rightarrow y$ and 
$x\rightsquigarrow x\wedge y=x\rightsquigarrow y$ (\cite{Dvu7}). \\
The following theorem provides a connection of pseudo equality algebras with the class of 
pseudo BCK(pC)-meet-semilattices. 
 
\begin{theo} \label{ps-eq-110} $\rm($\cite{Dvu7}$\rm)$ The following statements hold: \\
$(1)$  Let $\mathcal{A}=(A, \wedge, \thicksim, \backsim, 1)$ be a pseudo equality algebra.  
Then $\Psi(\mathcal{A})=(A, \wedge, \rightarrow, \rightsquigarrow, 1)$ is a pseudo BCK(pC)-meet-semilattice, where $x\rightarrow y=x\wedge y\thicksim x$ and $x\rightsquigarrow y=x\backsim x\wedge y$ for all $x, y\in A;$ \\ 
$(2)$ Let $\mathcal{B}=(B, \wedge, \rightarrow, \rightsquigarrow, 1)$ be a pseudo BCK(pC)-meet-semilattice. 
Then $\Phi(\mathcal{B})=(B, \wedge, \thicksim, \backsim, 1)$ is a pseudo equality algebra, where 
$x\thicksim y=y\rightarrow x$ and $x\backsim y=x\rightsquigarrow y$ for all $x, y\in B$.             
\end{theo}

With the notations of Theorem \ref{ps-eq-110} we say that a pseudo equality algebra 
$(A, \wedge, \thicksim, \backsim, 1)$ is \emph{invariant} if there exists a pseudo BCK-meet-semilattice 
$(A, \wedge, \rightarrow^{\prime}, \rightsquigarrow^{\prime}, 1)$ such that 
$\Phi((A, \wedge, \rightarrow^{\prime}, \rightsquigarrow^{\prime}, 1))=(A, \wedge, \thicksim, \backsim, 1)$. 

\begin{theo} \label{ps-eq-140} $\rm($\cite{Dvu7}$\rm)$ The following statements hold: \\
$(1)$ Let $(A, \wedge, \thicksim, \backsim, 1)$ be a pseudo equality algebra. 
Then $\Psi(\Phi(\Psi((A, \wedge, \rightarrow, \rightsquigarrow, 1))))=
\Psi((A, \wedge, \rightarrow, \rightsquigarrow, 1));$ \\
$(2)$ Let $(B, \wedge, \rightarrow, \rightsquigarrow, 1)$ be a pseudo BCK(pC)-meet-semilattice. Then 
$\Psi(\Phi((B, \wedge, \rightarrow, \rightsquigarrow, 1)))=(B, \wedge, \rightarrow, \rightsquigarrow, 1);$ \\
$(3)$ A pseudo equality algebra $(A, \wedge, \thicksim, \backsim, 1)$ is invariant if and only if 
$\Phi(\Psi((A, \wedge, \thicksim, \backsim, 1)))=(A, \wedge, \thicksim, \backsim, 1);$ \\
$(4)$ The class of pseudo BCK(pC)-meet-semilattices and the class of invariant pseudo equality  
algebras are term equivalent; \\
$(5)$ The category of pseudo BCK(pC)-meet-semilattices and the category of invariant pseudo equality  
algebras are categorically equivalent.
\end{theo}

\begin{prop} \label{ps-eq-150} $\rm($\cite{Ciu8}$\rm)$ A pseudo equality algebra $(A, \wedge, \thicksim, \backsim, 1)$ is invariant if and only if $x\wedge y\thicksim y=x\thicksim y$ and $x\backsim x\wedge y=x\backsim y$, 
for all $x, y\in A$.
\end{prop}

In what follows we recall some notions and results regarding the deductive systems and congruences on a 
pseudo equality algebra (see \cite{Dvu7}). \\
Let $(A, \wedge, \thicksim, \backsim, 1)$ be a pseudo equality algebra. A subset $D\subseteq A$ is called a \emph{deductive system} of $A$ if for all $x, y\in A$: \\
$(DS_1)$ $1 \in D;$ \\
$(DS_2)$ if $x\in D$, $x\le y$, then $y\in D$ (that is $D$ is an \emph{upset}); \\
$(DS_3)$ if $x, y \thicksim x \in D$, then $y\in D$. \\
A subset $D\subseteq A$ is a deductive system of $A$ if, for all $x, y\in A$, it satisfies conditions $(DS_1)$, $(DS_2)$ and the condition:\\
$(DS_3^{\prime})$ if $x, x \backsim y \in D$, then $y\in D$. \\
A deductive system $D$ of a pseudo equality algebra $A$ is \emph{proper} if $D\ne A$. 
A proper deductive system is called \emph{maximal} if it is not strictly contained in any other proper 
deductive system of $A$. 
We will denote by ${\mathcal DS}(A)$ the set of all deductive systems of $A$. 

\begin{lemma} \label{ps-eq-160} The following hold: \\
$(1)$ $D\in {\mathcal DS}(A)$ if and only if it satisfies conditions $(DS_1)$, $(DS_2)$ and the condition: \\
$(DS_4)$ $x, x\wedge y \thicksim x \in D$, then $y \in D$, for all $x, y\in A$. \\
$(2)$ $D\in {\mathcal DS}(A)$ if and only if it satisfies conditions $(DS_1)$, $(DS_2)$ and the condition: \\
$(DS^{'}_4)$ $x, x\backsim x\wedge y\in D$, then $y \in D$, for all $x, y\in A$.
\end{lemma} 
\begin{proof}
$(1)$ Assume that $D$ satisfies $(DS_1)$, $(DS_2)$, $(DS_3)$ and let $x, y\in A$ such that 
$x, x\wedge y \thicksim x \in D$. According to $(DS_3)$, it follows that $x\wedge y\in D$. Since $x\wedge y \le y$, 
by $(DS_2)$ we get $y\in D$, hence $(DS_4)$ is satisfied. 
Conversely, suppose that $D$ satisfies $(DS_1)$, $(DS_2)$, $(DS_4)$ and let $x, y\in A$ such that 
$x, y \thicksim x \in D$. By Proposition \ref{ps-eq-20}$(5)$, $y\thicksim x\le x\wedge y\thicksim x$, so by $(DS_2)$ 
we have $x\wedge y\thicksim x\in D$. Finally by $(DS_4)$, $y\in D$, thus $(DS_3)$ is satisfied. \\
$(2)$ Similarly as $(1)$.
\end{proof}

Clearly, $\{1\}, A \subseteq {\mathcal DS}(A)$ and ${\mathcal DS}(A)$ is closed under arbitrary intersections. 
As a consequence, $({\mathcal DS}(A), \subseteq)$ is a complete lattice. \\ 
Every deductive system of an invariant pseudo equality algebra $A$ is a subalgebra of $A$ (\cite{Ciu8}). \\
The set of deductive systems of an invariant  pseudo equality algebra coincides with the set of deductive systems of its corresponding pseudo BCK(pC)-meet-semilattice. \\
A deductive system $D$ of $A$ is called \emph{closed} if $x\thicksim y, x\backsim y\in D$ for all $x, y\in D$. 
According to \cite[Prop. 4.5]{Dvu7}, a deductive system $D$ of a pseudo equality algebra $A$ is closed 
if and only if $1\thicksim x, x\backsim 1\in D$ for all $x\in D$. \\
A deductive system $D$ of a pseudo equality algebra $A$ is called \emph{normal} if it satisfies the condition:\\
$(DS_5)$ $x \thicksim y, y\thicksim x \in D$ iff $y \backsim x, x\backsim y \in D$, for all $x, y\in A$. 
We will denote by ${\mathcal DS}_n(A)$ the set of all normal deductive systems of $A$. 
Obviously $\{1\}, A \in {\mathcal DS}_n(A)$ and ${\mathcal DS}_n(A)\subseteq {\mathcal DS}(A)$. \\
A subset $\Theta \subseteq A \times A$ is called a \emph{congruence} of $A$ if it is an 
equivalence relation on $A$ and for all $x_1, y_1, x_2, y_2\in A$ such that $(x_1, y_1), (x_2, y_2) \in \Theta$ 
the following hold: \\
$(CG_1)$ $(x_1 \wedge x_2, y_1 \wedge y_2) \in \Theta;$ \\
$(CG_2)$ $(x_1 \thicksim x_2, y_1 \thicksim y_2) \in \Theta;$ \\
$(CG_2)$ $(x_1 \backsim x_2, y_1 \backsim y_2) \in \Theta$. \\
We will denote by ${\mathcal Con}(A)$ the set of all congruences of $A$. \\
With any $H\in {\mathcal DS}_n(A)$ we associate a binary relation $\Theta_H$ by defining $x\Theta_H y$ iff 
$x \thicksim y \in H$ iff $x \backsim y \in H$. \\
If $\Theta$ is congruence relation on a pseudo 
equality algebra $(A, \wedge, \thicksim, \backsim, 1)$, then 
$F_{\Theta}=[1]_{\Theta}=\{x\in A\mid (x, 1)\in \Theta \}$ is a closed normal deductive system of $A$. \\
If $H\in {\mathcal DS}_n(A)$, then  
$H_{\Theta}=\{(x, y)\in A \times A \mid x\thicksim y, y\thicksim x \in H\}=
\{(x, y)\in A \times A \mid x\backsim y, y\backsim x \in H\} \in {\mathcal Con}(A)$.

Let $(A, \wedge, \thicksim, \backsim, 1)$ be an invariant pseudo equality algebra. 
Then there is a one-to-one correspondence between the set of all normal deductive systems of $A$ and 
${\mathcal Con}(A)$. 

Let $(A, \wedge, \thicksim, \backsim, 1)$ be a pseudo equality algebra and $H\in {\mathcal DS}_n(A)$. \\
Denote $A/{\Theta_H}=\{x/{\Theta_H}\mid x\in A\}$, where $x/{\Theta_H}=\{y\in A\mid (x,y)\in \Theta_H\}$. 
We define the following operations on $A/{\Theta_H}$: 
$x/{\Theta_H} \wedge y/{\Theta_H}=(x\wedge y)/{\Theta_H}$, 
$x/{\Theta_H} \thicksim y/{\Theta_H}=(x\thicksim y)/{\Theta_H}$, 
$x/{\Theta_H} \backsim y/{\Theta_H}=(x\backsim y)/{\Theta_H}$. \\
If $H\in {\mathcal DS}_n(A)$, then $(A/{\Theta_H},\wedge,\thicksim,\backsim,1/{\Theta_H})$ is a pseudo equality algebra. \\ 
A pseudo equality algebra $(A, \wedge, \thicksim, \backsim, 1)$ is called \emph{simple} if 
${\mathcal DS}(A)=\{\{1\}, A\}$.

$\vspace*{5mm}$

\section{Commutative pseudo equality algebras}

In this section we define and study the commutative pseudo equality algebras. 
We give a characterization of commutative pseudo equality algebras and we prove that an invariant 
pseudo equality algebra is commutative if and only if its corresponding pseudo BCK(pC)-meet-semilattice 
is commutative. Other results consist of proving that every commutative pseudo equality algebra is a distributive lattice and every finite invariant commutative pseudo equality algebra is a symmetric pseudo equality algebra.

\begin{Def} \label{com-eq-10} A pseudo equality algebra $(A, \wedge, \thicksim, \backsim, 1)$ is said to be \emph{commutative} if the following hold: \\
$\hspace*{3cm}$ $(x\wedge y\thicksim x)\backsim y=(x\wedge y\thicksim y)\backsim x,$ \\
$\hspace*{3cm}$ $y\thicksim (x\backsim x\wedge y)=x\thicksim (y\backsim x\wedge y)$ \\ 
for all $x, y\in A$.
\end{Def}

In other words, a pseudo equality algebra $A$ is commutative if and only if $x\vee_1 y=y\vee_1x$ and 
$x\vee_2 y=y\vee_2x$, for all $x, y\in A$.
Obviously an invariant pseudo equality algebra $(A, \wedge, \thicksim, \backsim, 1)$ is commutative if and only if $(y\thicksim x)\backsim y=(x\thicksim y)\backsim x$ and 
$y\thicksim (x\backsim y)=x\thicksim (y\backsim x)$, for all $x, y\in A$.

\begin{prop} \label{com-eq-20} An invariant pseudo equality algebra 
$(A, \wedge, \thicksim, \backsim, 1)$ is commutative if and only if its corresponding pseudo 
BCK(pC)-meet-semilattice $\Psi(A)$ is commutative.
\end{prop}
\begin{proof} 
Let $(A, \wedge, \thicksim, \backsim, 1)$ be an invariant pseudo equality algebra, so $\Phi(\Psi(A))=A$ and 
$x\wedge y\thicksim y=x\thicksim y$ and $x\backsim x\wedge y=x\backsim y$, for all $x, y\in A$. \\
Suppose that $(A, \wedge, \thicksim, \backsim, 1)$ is commutative, that is, 
$(x\wedge y\thicksim x)\backsim y=(x\wedge y\thicksim y)\backsim x,$ 
$y\thicksim (x\backsim x\wedge y)=x\thicksim (y\backsim x\wedge y)$ for all $x, y\in A$. \\
Applying Proposition \ref{ps-eq-20}$(1)$, it follows that \\
$\hspace*{2cm}$
$(x\wedge y\thicksim x)\backsim (x\wedge y\thicksim x)\wedge y=
(x\wedge y\thicksim y)\backsim (x\wedge y\thicksim y)\wedge x$ and \\
$\hspace*{2cm}$ 
$(x\backsim x\wedge y)\wedge y\thicksim (x\backsim x\wedge y)=
(y\backsim x\wedge y)\wedge x\thicksim (y\backsim x\wedge y)$ \\
for all $x, y\in A$. Hence: \\
$\hspace*{3cm}$ $(x\rightarrow y)\rightsquigarrow y=(y\rightarrow x)\rightsquigarrow x$ and \\
$\hspace*{3cm}$ $(x\rightsquigarrow y)\rightarrow y=(y\rightsquigarrow x)\rightarrow x$ \\
for all $x, y\in A$. \\ 
Thus $\Psi(A)=(A, \wedge, \rightarrow, \rightsquigarrow, 1)$ is a commutative pseudo BCK(pC)-meet-semilattice. \\
Conversely, suppose that $\Psi(A)$ is commutative. We have: \\ 
$\hspace*{2cm}$ 
$(x\wedge y\thicksim x)\backsim y=(y\thicksim x)\backsim y=(x\rightarrow y)\rightsquigarrow y=
(y\rightarrow x)\rightsquigarrow x$ \\
$\hspace*{4.7cm}$
$=(x\thicksim y)\backsim x=(x\wedge y\thicksim y)\backsim x)$ and \\
$\hspace*{2cm}$ 
$y\thicksim (x\backsim x\wedge y)=y\thicksim (x\backsim y)=(x\rightsquigarrow y)\rightarrow y=
(y\rightsquigarrow x)\rightarrow x$ \\
$\hspace*{4.7cm}$
$=x\thicksim (y\backsim x)=x\thicksim (y\backsim x\wedge y)$, \\
hence $\Phi(\Psi(A))=A$ is commutative. 
\end{proof}

\begin{cor} \label{com-eq-30} If $(A, \wedge, \thicksim, \backsim, 1)$ is a commutative pseudo equality algebra, 
then $(A, \le)$ is a distributive lattice.  
\end{cor} 
\begin{proof} If $A$ is commutative, then applying Proposition \ref{com-eq-20} it follows that $\Psi(A)$ is a 
commutative pseudo BCK-meet-semilattice. According to \cite[p. 65]{Kuhr6}, $(A, \le)$ is a join-semilattice lattice with $x\vee y=(x\rightarrow y)\rightsquigarrow y=(x\rightsquigarrow y)\rightarrow y$. 
Since $(A, \le)$ is a meet-semilattice, it folows that it is a lattice. 
Finally, applying \cite[Corollary 4.1.9]{Kuhr6}, $(A, \le)$ is a distributive lattice. 
\end{proof}

\begin{Def} \label{com-eq-40} A pseudo equality algebra $(A, \wedge, \thicksim, \backsim, 1)$ is said to be a 
\emph{symmetric pseudo equality algebra} if $x\backsim y=y\thicksim x$ for all $x, y\in A$. 
\end{Def}

Obviously any equality algebra is a symmetric equality algebra.

\begin{prop} \label{com-eq-50} Every finite invariant commutative pseudo equality algebra is a 
symmetric pseudo equality algebra. 
\end{prop}
\begin{proof} Let $(A, \wedge, \thicksim, \backsim, 1)$ be a finite commutative pseudo equality algebra. 
Applying Proposition \ref{com-eq-20}, $\Psi(A)=(A, \wedge, \rightarrow, \rightsquigarrow, 1)$ is a finite commutative pseudo BCK(pC) meet-semilattice. According to \cite[Corollary 4.1.6]{Kuhr6}, we have 
$x\rightarrow y= x\rightsquigarrow y$ for all $x, y\in \Psi(A)$. 
It follows that $x\backsim y=x\rightsquigarrow y=x\rightarrow y=y\thicksim x$ for all $x, y\in A$. 
Hence $A$ is a symmetric pseudo equality algebra. 
\end{proof}

\begin{rem} \label{com-eq-60} If $(A, \wedge, \thicksim, \backsim, 1)$ is a symmetric pseudo equality algebra, then $\Psi(A)=(A, \wedge, \rightarrow, \rightsquigarrow, 1)$ is a BCK(pC)-meet-semilattice. \\
Indeed, since $A$ is symmetric, we have $x\rightarrow y=x\wedge y\thicksim x=x\backsim x\wedge y=x\rightsquigarrow y$.
\end{rem}

\begin{theo} \label{com-eq-70} Let $(A, \wedge, \thicksim, \backsim, 1)$ be a pseudo equality algebra. 
Then the following are equivalent for all $x, y\in A$: \\
$(a)$ $A$ is commutative; \\
$(b)$ $x\wedge y\thicksim x=y\thicksim ((x\wedge y\thicksim y)\backsim x)$ and 
      $x\backsim x\wedge y=(x\thicksim (y\backsim x\wedge y))\backsim y;$ \\
$(c)$ $x\le y$ implies $y=(x\thicksim y)\backsim x=x\thicksim (y\backsim x)$.
\end{theo}
\begin{proof}
$(a)\Rightarrow (b)$ It follows by Proposition \ref{ps-eq-30} and $(a)$. \\
$(b)\Rightarrow (c)$ By $(b)$, for $x\le y$ we get $(x\thicksim (y\backsim x))\backsim y=1$.  
Applying Proposition \ref{ps-eq-10}$(3)$, it follows that $x\thicksim (y\backsim x)\le y$. 
On the other hand, by Proposition \ref{ps-eq-10}$(2)$ we have $y\le x\thicksim (y\backsim x)$, hence 
$y=x\thicksim (y\backsim x)$. Similarly $y=(x\thicksim y)\backsim x$. \\
$(c)\Rightarrow (a)$ By Proposition \ref{ps-eq-20}$(3)$ we have $x\le (x\wedge y\thicksim x)\backsim y$. 
Applying $(c)$ we get: \\
$\hspace*{1cm}$
$(x\wedge y\thicksim x)\backsim y = (x\thicksim ((x\wedge y\thicksim x)\backsim y))\backsim x$. \\
By Proposition \ref{ps-eq-20}$(4)$, $y\le (x\wedge y\thicksim x)\backsim y$. \\
Applying Proposition \ref{ps-eq-20-10}$(1)$ for $x:=y$, $y:=(x\wedge y\thicksim x)\backsim y$, $z:=x$, we have: \\ 
$\hspace*{1cm}$
$((x\wedge y\thicksim x)\backsim y)\wedge x\thicksim ((x\wedge y\thicksim x)\backsim y)\le x\wedge y\thicksim y$. \\ 
Since by Proposition \ref{ps-eq-20}$(3)$, $x\le (x\wedge y\thicksim x)\backsim y$, we get: \\
$\hspace*{1cm}$
$x\thicksim ((x\wedge y\thicksim x)\backsim y)\le x\wedge y\thicksim y$. \\
Applying again Proposition \ref{ps-eq-20-10}$(1)$ for $x:=x\thicksim ((x\wedge y\thicksim x)\backsim y)$, 
$y:=x\wedge y\thicksim y$, $z:=x$, we get: \\
$\hspace*{1cm}$
$(x\wedge y\thicksim y)\backsim (x\wedge y\thicksim y)\wedge x\le 
(x\thicksim ((x\wedge y\thicksim x)\backsim y))\backsim (x\thicksim ((x\wedge y\thicksim x)\backsim y))\wedge x$. \\
Hence by Propositions \ref{ps-eq-20}$(3)$ and \ref{ps-eq-30-10}$(1)$: \\ 
$\hspace*{1cm}$ 
$(x\wedge y\thicksim y)\backsim x \le (x\thicksim ((x\wedge y\thicksim x)\backsim y))\backsim x=
(x\wedge y\thicksim x)\backsim y$. \\
By interchanging $x$ and $y$ we obtain $(x\wedge y\thicksim x)\backsim y\le (x\wedge y\thicksim y)\backsim x$. \\
Thus $(x\wedge y\thicksim x)\backsim y = (x\wedge y\thicksim y)\backsim x$. 
Similarly $y\thicksim (x\backsim x\wedge y)=x\thicksim (y\backsim x\wedge y)$. \\
Hence $A$ is commutative. 
\end{proof} 

\begin{prop} \label{com-eq-70-10} The following statements hold: \\
$(1)$ if $(A, \wedge, \thicksim, \backsim, 1)$ is a commutative pseudo equality algebra, then $\Psi(A)$ is a commutative pseudo BCK(pC)-meet-semilattice; \\
$(2)$ if $(B, \wedge, \rightarrow, \rightsquigarrow, 1)$ is a commutative BCK(pC)-meet-semilattice, then $\Phi(B)$ 
is a commutative pseudo equality algebra.
\end{prop} 
\begin{proof} 
We recall that according to \cite[Th. 3.9]{Ciu7}, a pseudo BCK-algebra $(X, \rightarrow, \rightsquigarrow, 1)$ is 
commutative if and only if for all $x, y\in X$ such that $x\le y$, we have  
$y=(y\rightarrow x)\rightsquigarrow x=(y\rightsquigarrow x)\rightarrow x$. \\
$(1)$ Suppose that $(A, \wedge, \thicksim, \backsim, 1)$ is a commutative pseudo equality algebra and 
let $x, y\in A$ such that $x\le y$. It follows by Theorem \ref{com-eq-70} that 
$y=(x\thicksim y)\backsim x=x\thicksim (y\backsim x)$. Applying Proposition \ref{ps-eq-20}$(1)$ we get: \\
$\hspace*{2cm}$
$(y\rightarrow x)\rightsquigarrow x=(x\wedge y\thicksim y)\backsim (x\wedge y\thicksim y)\wedge x=
(x\thicksim y)\backsim x=y$ and \\ 
$\hspace*{2cm}$
$(y\rightsquigarrow x)\rightarrow x=(y\backsim x\wedge y)\wedge x\thicksim (y\backsim x\wedge y)=
x\thicksim (y\backsim x)=y$, \\
that is $\Psi(A)$ is commutative. \\
$(2)$ Let $(B, \wedge, \rightarrow, \rightsquigarrow, 1)$ be a commutative BCK(pC)-meet-semilattice and let 
$x, y\in B$ such that $x\le y$. We have $(x\thicksim y)\backsim x=(y\rightarrow x)\rightsquigarrow x=y$ and 
$x\thicksim (y\backsim x)=(y\rightsquigarrow x)\rightarrow x=y$. \\
Hence by Theorem \ref{com-eq-70}, $\Phi(B)$ is commutative.  
\end{proof}

\begin{ex} \label{ps-ex-80} 
Let $(A, \wedge, \rightarrow, \rightsquigarrow, 1)$ be a commutative pseudo BCK-meet-semilattice, that is  
$(x\rightarrow y)\rightsquigarrow y=(y\rightarrow x)\rightsquigarrow x$ and 
$(x\rightsquigarrow y)\rightarrow y=(y\rightsquigarrow x)\rightarrow x$ for all $x, y\in A$. \\
By \cite[Lemma 4.1.12]{Kuhr6}, $A$ is a commutative pseudo BCK(pD)-meet-semilattice, so 
$(A, \wedge, \thicksim, \backsim, 1)$ is a pseudo equality algebra, where 
$x\thicksim y=y\rightarrow x$, $x\backsim y=x\rightsquigarrow y$.  
Moreover, by Proposition \ref{com-eq-70-10} it follows that the pseudo equality algebra $A$ is commutative. 
\end{ex}

\begin{ex} \label{ps-ex-90}
Let $(G, \vee,\wedge, \cdot, ^{-1}, e)$ be an $\ell$-group. On the negative cone $G^{-}=\{g\in G \mid g\le e\}$ we define the operations $x\rightarrow y=y\cdot (x\vee y)^{-1}$, $x\rightsquigarrow y=(x\vee y)^{-1}\cdot y$. 
According to \cite[Example 4.1.2]{Kuhr6} and Example \ref{ps-ex-80} , 
$(G^{-}, \wedge, \rightarrow, \rightsquigarrow, e)$ is a commutative pseudo BCK(pD)-meet-semilattice. 
By Propositions \ref{ps-eq-50} and \ref{com-eq-70-10} it follows that $\Phi(G^{-})$ is a commutative pseudo equality algebra.
\end{ex}

\begin{ex} \label{ps-ex-10} $\rm($\cite{Dvu7}$\rm)$
Let $(G, \vee,\wedge, \cdot, ^{-1}, e)$ be an $\ell$-group. On the negative cone $G^{-}=\{g\in G \mid g\le e\}$ we define the operations 
$x\thicksim y=(x\cdot y^{-1})\wedge e$, $x\backsim y=(x^{-1}\cdot y)\wedge e$.  
Then $(G^{-}, \wedge, \thicksim, \backsim, e)$ is a pseudo equality algebra. 
We have $x\thicksim y=y\backsim x$ if and only if $G$ is Abelian. 
Thus $(G^{-}, \wedge, \thicksim, \backsim, e)$ is a symmetric pseudo equality algebra if and only if $G$ is Abelian. 
\end{ex}

\begin{ex}\label{ps-ex-50} $\rm($\cite{Ciu8}$\rm)$
Let $A=\{0,a,b,1\}$ with $0<a, b<1$ be a lattice whose diagram is below. 
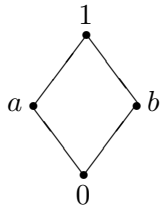
\begin{center}
\begin{picture}(50,100)(0,0)
\put(37,11){\circle*{3}}
\put(34,0){$0$}
\put(37,11){\line(3,4){20}}
\put(57,37){\circle*{3}}
\put(61,35){$b$}

\put(37,11){\line(-3,4){20}}
\put(18,37){\circle*{3}}
\put(8,35){$a$}

\put(18,37){\line(3,4){20}}
\put(38,64){\circle*{3}}
\put(35,68){$1$} 

\put(38,64){\line(3,-4){20}}

\end{picture}
\end{center}

Then $(A, \wedge, \rightarrow, 1)$ is a BCK(C)-lattice, and $\Phi(A)=(A, \wedge, \thicksim, \backsim, 1)$ 
is a pseudo equality algebra with the operations $\thicksim, \backsim$ given below: 
\[
\hspace{10mm}
\begin{array}{c|ccccc}
\thicksim& 0 & a & b & 1 \\ \hline
0 & 1 & b & a & 0 \\ 
a & 1 & 1 & a & a \\ 
b & 1 & b & 1 & b \\  
1 & 1 & 1 & 1 & 1
\end{array}
\hspace{10mm}
\begin{array}{c|ccccc}
\backsim& 0 & a & b & 1 \\ \hline
0 & 1 & 1 & 1 & 1 \\ 
a & b & 1 & b & 1 \\ 
b & a & a & 1 & 1 \\  
1 & 0 & a & b & 1
\end{array}
.
\]
One can easily chack that $\Phi(\Psi(A))=A$, thus $(A, \wedge, \thicksim, \backsim, 1)$ is an invariant pseudo 
equality algebra. Moreover, $A$ is a commutative and symmetric pseudo equality algebra. 
We mention that ${\mathcal DS}(A)={\mathcal DS}_n(A)=\{\{1\}, \{a,1\}, \{b,1\}, A\}$. 
\end{ex}


$\vspace*{5mm}$

\section{Commutative deductive systems of pseudo equality algebras}

We introduce the notion of a commutative deductive system of a pseudo equality algebra and we give equivalent  
conditions for this notion. We show that a pseudo equality algebra $A$ is commutative if and only if $\{1\}$ is a 
commutative deductive system of $A$. 
Another result consists of proving that all deductive systems of a commutative pseudo equality algebra are commutative. 
It is also proved that a normal deductive system $H$ of a pseudo equality algebra $A$ is commutative if and only if $A/H$ is a commutative pseudo equality algebra.

\begin{Def} \label{com-ds-10} A deductive system $D$ of a pseudo equality algebra $(A, \wedge, \thicksim, \backsim, 1)$ is said to be \emph{commutative} if it satisfies the following conditions for all $x, y\in A$: \\
$(cds_1)$ $x\wedge y\thicksim y\in D$ implies $x\thicksim ((x\wedge y\thicksim x)\backsim y)\in D,$ \\
$(cds_2)$ $y\backsim x\wedge y\in D$ implies $(y\thicksim (x\backsim x\wedge y))\backsim x\in D$.  
\end{Def}

We will denote by ${\mathcal DS}_c(A)$ the set of all commutative deductive systems of $A$. \\
In other words, $D\in {\mathcal DS}_c(A)$ if and only if $x\wedge y\thicksim y\in D$ implies 
$x\thicksim x\vee_1 y\in D$ and  $y\backsim x\wedge y\in D$ implies $x\vee_2 y\backsim x\in D$, 
for all $x, y\in A$. 

\begin{prop} \label{com-ds-20} If $A$ is a commutative pseudo equality algebra, then 
${\mathcal DS}(A)={\mathcal DS}_c(A)$. 
\end{prop}
\begin{proof}
By Theorem \ref{com-eq-70}, $x\wedge y\thicksim y=x\thicksim ((x\wedge y\thicksim x)\backsim y)$ and 
$y\backsim x\wedge y=(y\thicksim (x\backsim x\wedge y))\backsim x$, for all $x, y\in A$. 
If $D\in {\mathcal DS}(A)$, then $x\wedge y\thicksim y\in D$ iff $x\thicksim ((x\wedge y\thicksim x)\backsim y)\in D$ 
and $y\backsim x\wedge y\in D$ iff $(y\thicksim (x\backsim x\wedge y))\backsim x\in D$, for all $x, y\in A$. 
Hence $D\in {\mathcal DS}_c(A)$. 
\end{proof}

\begin{theo} \label{com-ds-30} An upset $D$ of a pseudo equality algebra $A$ is a commutative deductive system of $A$ if and only if it satisfies the following conditions for all $x, y, z\in A$: \\
$(1)$ $1\in D;$ \\ 
$(2)$ $(x\wedge y\thicksim y)\wedge z\thicksim z, z\in D$ implies 
      $x\thicksim ((x\wedge y\thicksim x)\backsim y)\in D;$ \\
$(3)$ $z\backsim (y\backsim x\wedge y)\wedge z, z\in D$ implies 
      $(y\thicksim (x\backsim x\wedge y))\backsim x\in D$.
\end{theo}
\begin{proof} 
Let $D\subseteq A$ be an upset of $A$ satisfying conditions $(1)$, $(2)$ and $(3)$. \\
By $(1)$, $1\in D$, that is $(DS_1)$. \\
Consider $x, y\in A$ such that $x, x\wedge y\thicksim x\in D$. 
We have $(y\wedge 1\thicksim 1)\wedge x\thicksim x=x\wedge y\thicksim x\in D$ and applying $(2)$ we get 
$y=y\thicksim ((y\wedge 1\thicksim y)\backsim 1)\in D$, that is $(DS_4)$. \\
Since $D$ is an upset, then $(DS_2)$ is satisfied. 
Hence $D\in {\mathcal DS}(A)$. \\ 
Let $x, y\in A$ such that $x\wedge y\thicksim y\in D$. \\
Since $(x\wedge y\thicksim y)\wedge 1\thicksim 1=x\wedge y\thicksim y\in D$ and $1\in D$, 
applying $(2)$ we get $(cds_1)$. \\
Similarly, if $y\backsim x\wedge y\in D$ we have $1\backsim (y\backsim x\wedge y)\wedge 1=y\backsim x\wedge y\in D$ 
and $1\in D$, from $(3)$ we get $(cds_2)$. 
Thus $D\in {\mathcal DS}_c(A)$. \\
Conversely, assume that $D\in {\mathcal DS}_c(A)$. Since $1\in D$, condition $(1)$ is satisfied. \\
Consider $x, y, z\in D$ such that $(x\wedge y\thicksim y)\wedge z\thicksim z, z\in D$, so 
$(x\wedge y\thicksim y)\wedge z\in D$. \\
Since $(x\wedge y\thicksim y)\wedge z\le x\wedge y\thicksim y$, we get $x\wedge y\thicksim y\in D$. 
Applying $(cds_1)$, it follows that $x\thicksim ((x\wedge y\thicksim x)\backsim y)\in D$, that is $(2)$. \\
Similarly from $z\backsim (y\backsim x\wedge y)\wedge z, z\in D$ we get $(3)$.  
\end{proof}

\begin{prop} \label{com-ds-40} A pseudo equality algebra $A$ is commutative if and only if 
$\{1\}\in {\mathcal DS}_c(A)$. 
\end{prop} 
\begin{proof}
If $A$ is commutative, then by Proposition \ref{com-ds-20}, $\{1\}\in {\mathcal DS}_c(A)$. \\
Conversely, assume that $\{1\}\in {\mathcal DS}_c(A)$ and let $x, y\in A$ such that $y\le x$, that is 
$x\wedge y\thicksim y=1\in \{1\}$. 
It follows that $x\thicksim ((x\wedge y\thicksim x)\backsim y)\in \{1\}$, hence  
$x\thicksim ((x\wedge y\thicksim x)\backsim y)=1$, that is $x\thicksim ((y\thicksim x)\backsim y)=1$. 
Applying Proposition \ref{ps-eq-10}$(3)$ we get $(y\thicksim x)\backsim y\le x$. 
Since by Proposition \ref{ps-eq-20}$(3)$, $x\le (y\thicksim x)\backsim y$, we get $x=(y\thicksim x)\backsim y$. 
Similarly $x=y\thicksim (x\backsim y)$. \\
Hence by Theorem \ref{com-eq-70}, $A$ is commutative. 
\end{proof}

\begin{cor} \label{com-ds-40-10} A pseudo equality algebra $A$ is commutative if and only if 
${\mathcal DS}(A)={\mathcal DS}_c(A)$. 
\end{cor}

\begin{theo} \label{com-ds-50} Let $A$ be a pseudo equality algebra and $H\in {\mathcal DS}_n(A)$. 
Then $H\in {\mathcal DS}_c(A)$ if and only if $A/H$ is a commutative pseudo equality algebra.
\end{theo}
\begin{proof}
Assume $H\in {\mathcal DS}_c(A)$ and let $x, y\in A$ such that 
$[x]_{\Theta_H}\wedge [y]_{\Theta_H}\thicksim [y]_{\Theta_H}\in \{[1]_{\Theta_H}\}$, that is 
$[x\wedge y\thicksim y]_{\Theta_H}=[1]_{\Theta_H}$, so $x\wedge y\thicksim y\in H$. 
Since $H$ is commutative, we get $x\thicksim ((x\wedge y\thicksim x)\backsim y)\in H$, so that 
$[x]_{\Theta_H}\thicksim (([x]_{\Theta_H}\wedge [y]_{\Theta_H}\thicksim [x]_{\Theta_H})\backsim [y]_{\Theta_H})= 
[x\thicksim ((x\wedge y\thicksim x)\backsim y)]_{\Theta_H}=[1]_{\Theta_H}\in \{[1]_{\Theta_H}\}$. 
Similarly from $[y]_{\Theta_H}\backsim [x]_{\Theta_H}\wedge [y]_{\Theta_H}\in \{[1]_{\Theta_H}\}$ we get 
$([y]_{\Theta_H}\thicksim ([x]_{\Theta_H}\wedge [y]_{\Theta_H}))\backsim [x]_{\Theta_H}\in \{[1]_{\Theta_H}\}$. 
Hence $\{[1]_{\Theta_H}\}\in {\mathcal DS}_c(A/H)$, so by Proposition \ref{com-ds-40}, $A/H$ is a commutative 
pseudo equality algebra. \\
Conversely, if $A/H$ is a commutative pseudo equality algebra, then by Proposition \ref{com-ds-40}, 
$\{[1]_{\Theta_H}\}\in {\mathcal DS}_c(A/H)$.  
If $x\wedge y\thicksim y\in H=[1]_{\Theta_H}$, we have 
$[x]_{\Theta_H}\wedge [y]_{\Theta_H}\thicksim [y]_{\Theta_H}\in \{[1]_{\Theta_H}\}$. 
Since $\{[1]_{\Theta_H}\}\in {\mathcal DS}_c(A/H)$, we get 
$[x]_{\Theta_H}\thicksim (([x]_{\Theta_H}\wedge [y]_{\Theta_H}\thicksim [x]_{\Theta_H})\backsim [y]_{\Theta_H}\in 
\{[1]_{\Theta_H}\}$, so $[x\thicksim ((x\wedge y\thicksim x)\backsim y)]_{\Theta_H}=[1]_{\Theta_H}$, that is 
$x\thicksim ((x\wedge y\thicksim x)\backsim y)\in H$. \\
Similarly from $y\backsim x\wedge y\in H$ we get $y\thicksim (x\backsim x\wedge y)\backsim x)\in H$. 
Hence $H\in {\mathcal DS}_c(A)$. 
\end{proof}

\begin{prop} \label{com-ds-60} Let $A$ be a pseudo equality algebra and $D\in {\mathcal DS}_c(A)$. 
Then the following hold for all $x, y\in A$: \\
$(1)$ $x\thicksim ((y\thicksim x)\backsim y), (y\thicksim (x\backsim y))\backsim x \in D$, whenever $y\le x;$ \\
$(2)$ $((x\thicksim y)\backsim x)\thicksim ((y\thicksim ((x\thicksim y)\backsim x))\backsim y),  
       (y\thicksim (x\thicksim (y\backsim x)))\backsim (x\thicksim (y\backsim x))\in D$. 
\end{prop}
\begin{proof}
$(1)$ Since $y\le x$, then $x\wedge y\thicksim y=y\backsim x\wedge y=1\in D$, hence 
$x\thicksim ((y\thicksim x)\backsim y)\in D$ and $(y\thicksim (x\backsim y))\backsim x \in D$. \\
$(2)$ It follows by $(1)$, since by Proposition \ref{ps-eq-10} we have $y\le (x\thicksim y)\backsim x)$ and 
$y\le x \thicksim (y\backsim x)$.  
\end{proof}

\begin{cor} \label{com-ds-70} Let $A$ be a commutative pseudo equality algebra and $x, y\in A$ such that $y\le x$. 
Then $x\thicksim ((y\thicksim x)\backsim y)=(y\thicksim (x\backsim y))\backsim x=1$. 
\end{cor} 
\begin{proof}
Since $A$ is commutative, we have $\{1\}\in {\mathcal DS}_c(A)$ and by Proposition \ref{com-ds-60} we get 
$x\thicksim ((y\thicksim x)\backsim y), (y\thicksim (x\backsim y))\backsim x \in \{1\}$, that is 
$x\thicksim ((y\thicksim x)\backsim y)=(y\thicksim (x\backsim y))\backsim x=1$.
\end{proof}

\begin{cor} \label{com-ds-80} Let $A$ be a commutative pseudo equality algebra and $x, y\in A$ such that $y\le x$. 
Then $x=(y\thicksim x)\backsim y=y\thicksim (x\backsim y)$. 
\end{cor} 
\begin{proof}
By Corollary \ref{com-ds-70} and Proposition \ref{ps-eq-10}$(4)$ we get $(y\thicksim x)\backsim y\le x$ and 
$y\thicksim (x\backsim y)\le x$. 
Applying Proposition \ref{ps-eq-10}$(2)$ we have $x\le(y\thicksim x)\backsim y$ and $x\le y\thicksim (x\backsim y)$. 
Hence $x=(y\thicksim x)\backsim y=y\thicksim (x\backsim y)$.
\end{proof}

\begin{cor} \label{com-ds-90} Let $A$ be a commutative pseudo equality algebra and $x, y\in A$ such that 
$x\thicksim y=1$ or $y\backsim x=1$. Then $x=(y\thicksim x)\backsim y=y\thicksim (x\backsim y)$. 
\end{cor} 
\begin{proof}
It is a consequence of Proposition \ref{ps-eq-10}$(3)$ and Corollary \ref{com-ds-80}.
\end{proof}

$\vspace*{5mm}$

\section{Measures and internal states on pseudo equality algebras}

As application of the results proved in the previous sections, we define and study the measures and 
measure-morphisms on pseudo equality algebras, and we prove new properties of state pseudo equality algebras. 
We show that any measure-morphism on a pseudo equality algebra is also a measure, and that the kernel of a 
measure is a commutative deductive system. We prove that the quotient pseudo equality algebra over the kernel 
of a measure is a commutative pseudo equality algebra. 
It is also proved that a pseudo equality algebra possessing an order-determining system is commutative. 
Other main results consist of proving that the measures on a pseudo equality algebra and the and measure-morphisms on 
a linearly ordered pseudo equality algebra coincide with those on its corresponding pseudo BCK(pC)-meet semilattice.  
We prove that the two types of internal states on a pseudo equality algebra coincide if and 
only if it is a commutative pseudo equality algebra.
Moreover if the pseudo equality algebra is symmetric and linearly ordered, then the state-morphisms coincide 
with the two types of internal states. 

\begin{Def} \label{ms-eq-10} Let $(A, \wedge, \thicksim, \backsim, 1)$ be a pseudo equality algebra. 
A map $m:A\longrightarrow [0, +\infty)$ is said to be a \emph{measure} if 
$m(y\thicksim x)=m(x\backsim y)=m(y)-m(x)$, whenever $y\le x$. \\
If $m(x\wedge y\thicksim x)=m(x\backsim x\wedge y)=\max\{0,m(y)-m(x)\}$, then 
$m$ is said to be a \emph{measure-morphism}.
\end{Def}

Denote $\mathcal{M}_{EQA}(A)$ the set of all measures and $\mathcal{MM}_{EQA}(A)$ the set of all 
measure-morphisms on a pseudo equality algebra $A$. 

\begin{prop} \label{ms-eq-20} Let $m$ be a measure on a pseudo equality algebra $A$. Then the following hold 
for all $x, y\in A$:\\
$(1)$ $m(1)=0;$ \\
$(2)$ $m(x)\ge m(y)$ whenever $x\le y;$ \\
$(3)$ $m(x\vee_1 y)=m(y\vee_1 x)$ and $m(x\vee_2 y)=m(y\vee_2 x);$ \\
$(4)$ $m(x\vee_1 y)=m(y\vee_2 x);$ \\ 
$(5)$ $m(x\wedge y\thicksim x)=m(x\backsim x\wedge y)$.
\end{prop}
\begin{proof}
$(1)$ From $1\le 1$ we have $m(1)=m(1\thicksim 1)=m(1)-m(1)=0$. \\
$(2)$ Since $x\le y$, it follows that $m(x)-m(y)=m(x\thicksim y)\ge 0$, so that $m(x)\ge m(y)$. \\
$(3)$ First, let $x\le y$. By Proposition \ref{ps-eq-30-20}$(2)$ we have $m(x\vee_1 y)=m(y)$. \\
Since $x\le x\wedge y\thicksim y$ and $x\wedge y\le y$, we have: \\
$\hspace*{1cm}$ 
$m(y\vee_1 x)=m((x\wedge y\thicksim y)\backsim x)=m(x)-m(x\wedge y\thicksim y)$ \\
$\hspace*{2.8cm}$
$=m(x)-m(x\wedge y)+m(y)=m(y)$, \\
hence $m(x\vee_1 y)=m(y\vee_1 x)$. \\
Let now $x, y\in A$ be arbitrary. Since by Proposition \ref{ps-eq-30-20}$(5)$ 
$y\le x\vee_1 y$ implies $y\vee_1 x\le (x\vee_1 y)\vee_1 x$ and 
$x\le y\vee_1 x$ implies $x\vee_1 y\le (y\vee_1 x)\vee_1 y=m(x\vee_1 y)$, 
applying the first part of the proof and the properties of measures we get: \\ 
$\hspace*{1cm}$
$m(x\vee_1 y)=m(x\vee_1 (x\vee_1 y))=m((x\vee_1 y)\vee_1 x)\le m(y\vee_1 x)$ \\
$\hspace*{2.8cm}$
$=m(y\vee_1 (y\vee_1 x))=m((y\vee_1 x)\vee_1 y)\le m(x\vee_1 y)$, \\
thus $m(x\vee_1 y)=m(y\vee_1 x)$. Similarly $m(x\vee_2 y)=m(y\vee_2 x)$. \\
$(4)$ First, let $x\le y$. Applying Proposition \ref{ps-eq-30-20}$(2)$ we get $m(x\vee_1 y)=m(x\vee_2 y)=m(y)$. \\ 
Let now $x, y\in A$ be arbitrary. By $(3)$ we have: \\
$\hspace*{1cm}$ 
$m(x\vee_1 y)=m(x\vee_1 (x\vee_1 y))=m(x\vee_2 (x\vee_1 y))=m((x\vee_1 y)\vee_2 x)$ \\
$\hspace*{2.8cm}$
$\le m(y\vee_2 x)=m(x\vee_2 y)=m(x\vee_2 (x\vee_2 y))=m((x\vee_2 y)\vee_2 x)$ \\
$\hspace*{2.8cm}$
$=m((x\vee_2 y)\vee_1 x)\le m(y\vee_1 x)=m(x\vee_1 y)$. \\
Thus $m(x\vee_1 y)=m(x\vee_2 y)$. \\
$(5)$ By Propositions \ref{ps-eq-30} and \ref{ps-eq-20}$(4)$ we get: \\
$\hspace*{1cm}$
$m(x\wedge y\thicksim x)=m(y\thicksim ((x\wedge y\thicksim x)\backsim y)=m(y)-m((x\wedge y\thicksim x)\backsim y)$ \\
$\hspace*{3.3cm}$
$=m(y)-m(x\vee_1 y)=m(y)-m(x\vee_2 y)$ \\
$\hspace*{3.3cm}$
$=m(x\vee_2 y\backsim y)=m(x\backsim x\wedge y)$.  
\end{proof}

For $m\in \mathcal{M}_{EQA}(A)$, $\Ker(m)=\{x\in A\mid m(x)=0\}$ is called the \emph{kernel} of $m$. \\
By Proposition \ref{ms-eq-20}$(1)$, $1\in \Ker(m)$. 

\begin{prop} \label{ms-eq-30} Let $A$ be a pseudo equality algebra and let $m$ be a measure on $A$. Then: \\
$(1)$ if $x, y\in A$ such that $y\le x$, then $m(x\vee_1 y)=m(x\vee_2 y)=m(x);$ \\
$(2)$ $\Ker(m)\in {\mathcal DS}(A);$ \\
$(3)$ if $A$ is invariant, then $\Ker(m)\in {\mathcal DS}_n(A)$. 
\end{prop}
\begin{proof}
$(1)$ By Proposition \ref{ps-eq-30-20} we have $y\vee_1 x=y\vee_2 x=x$, and applying Proposition \ref{ms-eq-20}$(3)$ 
we get: $m(x\vee_1 y)=m(y\vee_1 x)=m(x)$ and $m(x\vee_2 y)=m(y\vee_2 x)=m(x)$. \\ 
$(2)$ Obviously $1\in \Ker(m)$. \\ 
Consider $x, y\in A$ such that $x\le y$ and $x\in \Ker(m)$, that is $m(x)=0$. 
By Proposition \ref{ms-eq-20}$(2)$, $0=m(x)\ge m(y)$, hence $m(y)=0$, so $y\in \Ker(m)$. \\
Let $x, y\in A$ such that $x, y\thicksim x\in \Ker(m)$. 
Since $x\le x\vee_1 y$, we have $0=m(x)\ge m(x\vee_1 y)$, thus $m(x\vee_1 y)=0$. 
By Proposition \ref{ps-eq-20}$(5)$, $y\thicksim x\le x\wedge y\thicksim x$, so 
$0=m(y\thicksim x)\ge m(x\wedge y\thicksim x)$, that is $m(x\wedge y\thicksim x)=0$. 
Applying Proposition \ref{ps-eq-30}, $0=m(x\wedge y\thicksim x)=m(y\thicksim x\vee_1 y)=m(y)-m(x\vee_1 y)=m(y)$. 
Hence $y\in \Ker(m)$. We conclude that $\Ker(m)\in {\mathcal DS}(A)$. \\
$(3)$ Let $x, y\in A$ such that $x\thicksim y\in \Ker(m)$.   
Applying Propositions \ref{ps-eq-150} and \ref{ms-eq-20}$(5)$ we have: 
$0=m(x\thicksim y)=m(x\wedge y\thicksim y)=m(y\backsim x\wedge y)=m(y\backsim x)$, that is $y\backsim x\in \Ker(m)$. \\
Similarly $y\thicksim x\in \Ker(m)$ implies $x\backsim y\in \Ker(m)$. \\
Conversely, consider $x, y\in A$ such that $y\backsim x\in \Ker(m)$. 
Similarly as above we have: \\
$0=m(y\backsim x)=m(y\backsim x\wedge y)=m(x\wedge y\thicksim y)=m(x\thicksim y)$, hence $x\thicksim y\in \Ker(m)$. \\
Similarly $x\backsim y\in \Ker(m)$ implies $y\thicksim x\in \Ker(m)$. \\
We conclude that $\Ker(m)\in {\mathcal DS}_n(A)$.
\end{proof}

\begin{prop} \label{ms-eq-40} If $A$ is a pseudo equality algebra, then 
$\mathcal{MM}_{EQA}(A)\subseteq \mathcal{M}_{EQA}(A)$. 
\end{prop}
\begin{proof}
Let $m\in \mathcal{MM}_{EQA}(A)$. We have $m(1)=m(1\wedge 1\thicksim 1)=\max\{0, m(1)-m(1)\}=0$. \\
Consider $x, y\in A$ such that $y\le x$. Then we have: 
$0=m(1)=m(y\wedge x\thicksim y)=\max\{0, m(x)-m(y)\}$, hence $m(x)\le m(y)$. 
It follows that $m(y\thicksim x)=m(x\wedge y\thicksim x)=\max\{0, m(y)-m(x)\}=m(y)-m(x)$. 
Similarly $m(x\backsim y)=m(y)-m(x)$. \\
We conclude that $m\in \mathcal{M}_{EQA}(A)$, that is $\mathcal{MM}_{EQA}(A)\subseteq \mathcal{M}_{EQA}(A)$.
\end{proof}

\begin{ex} \label{ms-eq-50} Let $(A, \wedge, \thicksim, \backsim, 1)$ be the pseudo equality algebra from 
Example \ref{ps-ex-50}. Define $u_{\alpha,\beta}:A\longrightarrow [0, +\infty)$ by $u_{\alpha,\beta}(0)=\alpha$, 
$u_{\alpha,\beta}(a)=\beta$, $u_{\alpha,\beta}(b)=\alpha-\beta$, $u_{\alpha,\beta}(1)=0$, with 
$\alpha\ge \beta\ge 0$. 
Then $\mathcal{M}_{EQA}(A)=\mathcal{MM}_{EQA}(A)=\{u_{\alpha,\beta}\mid \alpha\ge \beta\ge 0\}$. 
\end{ex}

\begin{prop} \label{ms-eq-60} Let $A$ be a pseudo equality algebra and $m\in \mathcal{M}_{EQA}(A)$. 
Then $\Ker(m)\in \mathcal{DS}_c(A)$. 
\end{prop}
\begin{proof}
Let $x, y\in A$ such that $x\wedge y\thicksim y\in \Ker(m)$, that is $m(x\wedge y\thicksim y)=0$.  
By Propositions \ref{ps-eq-20}$(3)$, \ref{ms-eq-20}$(3)$, \ref{ps-eq-20}$(1)$ we have 
$x\le ((x\wedge y\thicksim x)\backsim y)$, $m(((x\wedge y\thicksim x)\backsim y)=m((x\wedge y\thicksim y)\backsim x)$, 
and $x\le x\wedge y\thicksim y$, respectively. Applying the definition of $m$ we get: \\
$\hspace*{1cm}$
$m(x\thicksim ((x\wedge y\thicksim x)\backsim y))=m(x)-m((x\wedge y\thicksim x)\backsim y)$ \\
$\hspace*{5.4cm}$
$=m(x)-m((x\wedge y\thicksim y)\backsim x)$ \\
$\hspace*{5.4cm}$
$=m(x)-m(x)+m(x\wedge y\thicksim y)=0$, \\
hence $x\thicksim ((x\wedge y\thicksim x)\backsim y)\in \Ker(m)$. \\
Similarly, if $x, y\in A$ such that $y\backsim x\wedge y\in \Ker(m)$, then $m(y\backsim x\wedge y)=0$, and we have: \\
$\hspace*{1cm}$
$m((y\thicksim (x\backsim x\wedge y))\backsim x)=m(x)-m(y\thicksim (x\backsim x\wedge y))$ \\
$\hspace*{5.4cm}$
$=m(x)-m(x\thicksim (y\backsim x\wedge y))$ \\
$\hspace*{5.4cm}$
$=m(x)-m(x)+m(y\backsim x\wedge y))=0$. \\
Thus $y\thicksim (x\backsim x\wedge y))\backsim x\in \Ker(m)$, and we conclude that $\Ker(m)\in \mathcal{DS}_c(A)$. 
\end{proof}

\begin{theo} \label{ms-eq-70} Let $A$ be an invariant pseudo equality algebra and $m\in \mathcal{M}_{EQA}(A)$ and 
let $\hat{m}:A/\Ker(m)\longrightarrow [0, +\infty)$ defined by $\hat{m}(\hat{x}):=m(x)$, 
where $\hat{x}=x/\Ker(m)\in A/\Ker(m)$. Then $A/\Ker(m)$ is a commutative pseudo equality algebra and 
$\hat{m}\in \mathcal{M}_{EQA}(A/\Ker(m))$.   
\end{theo}
\begin{proof} 
According to Propositions \ref{ms-eq-30} and \ref{ms-eq-60}, $\Ker(m)\in {\mathcal DS}_n(A)$ and 
$\Ker(m)\in {\mathcal DS}_c(A)$. By Theorem \ref{com-ds-50}, $A/\Ker(m)$ is a commutative pseudo 
equality algebra. \\
We show that $\hat{m}$ is a well-defined function on $A/\Ker(m)$. Indeed, let $x,y \in A$ such that 
$\hat{x}=\hat{y}$, that is $x\thicksim y, y\thicksim x\in A/\Ker(m)$. 
Since $A$ is invariant, applying Propositin \ref{ps-eq-30-20}$(6)$ we have: \\
$\hspace*{1cm}$
$0=m(x\thicksim y)=m(x\wedge y\thicksim y)=m(x\thicksim y\vee_1 x)=m(x)-m(y\vee_1 x)$ and \\
$\hspace*{1cm}$
$0=m(y\thicksim x)=m(x\wedge y\thicksim x)=m(y\thicksim x\vee_1 y)=m(y)-m(x\vee_1 y)$. \\
By Proposition \ref{ms-eq-20}$(3)$, $m(x\vee_1 y)=m(y\vee_1 x)$, hence $m(x)=m(y)$, that is $\hat{m}$ is a well-defined. \\ 
Let now $\hat{x}, \hat{y}\in A/\Ker(m)$ such that $\hat{y}\le \hat{x}$, so by Proposition \ref{ps-eq-30-20}$(2)$, 
$\hat{y}\vee_1 \hat{x}=\hat{x}$. 
We have $\hat{m}(\hat{y}\thicksim \hat{x})=m(y\thicksim x)=m(x\wedge y\thicksim x)=m(y\thicksim x\vee_1 y)= 
m(y)-m(x\vee_1 y)$. \\
But $m(x\vee_1 y)=m(y\vee_1 x)=\hat{m}(\hat{y}\vee_1 \hat{x})=\hat{m}(\hat{x})=m(x)$.  
Hence $\hat{m}(\hat{y}\thicksim \hat{x})=\hat{m}(\hat{y})-\hat{m}(\hat{x})$. 
Similarly $\hat{m}(\hat{x}\backsim \hat{y})=\hat{m}(\hat{y})-\hat{m}(\hat{x})$. 
We conclude that $\hat{m}\in \mathcal{M}_{EQA}(A/\Ker(m))$.
\end{proof}

\begin{Def} \label{ms-eq-80} Let $A$ be a pseudo equality algebra. A system ${\mathcal S}$ of measures 
on $A$ is an \emph{order-determining system} on $A$ if for all measures $m\in {\mathcal S}$, 
$m(x)\ge m(y)$ implies $x\le y$.
\end{Def}

\begin{ex} \label{ms-eq-80-10} In Example \ref{ms-eq-50}, ${\mathcal S}=\{u_{\alpha,\beta}\mid \alpha\ge \beta\ge 0\}$ 
is an order-determining system on the pseudo equality algebra $A$. 
\end{ex}

\begin{prop} \label{ms-eq-90} If a pseudo equality algebra $A$ possesses an order-determining system ${\mathcal S}$ 
of measures, then $A$ is commutative. 
\end{prop}
\begin{proof} 
Let $m\in {\mathcal S}$ and let $x, y\in A$ such that $x\le y$. By Proposition \ref{ms-eq-30}$(1)$, 
$m((x\thicksim y)\backsim x)=m(x\thicksim (y\backsim x))=m(y)$. 
Since ${\mathcal S}$ is order-determining then $(x\thicksim y)\backsim x=x\thicksim (y\backsim x)=y$.  
Applying Theorem \ref{com-eq-70}, it follows that $A$ is commutative.
\end{proof}

Measures and measure-morphisms on pseudo BCK-algebras were introduced and studied in \cite{Ciu6}. 
Given a pseudo BCK-algebra $(B,\rightarrow, \rightsquigarrow,1)$, a measure on $B$ is a map 
$m:B\longrightarrow [0,+\infty)$ satisfying the axiom  
$m(x\rightarrow y)=m(x\rightsquigarrow y)=m(y)-m(x)$, whenever $y\le x$.  \\
A measure-morphism on $B$ is a map $m:B\longrightarrow [0,+\infty)$ satisying the axiom  
$m(x\rightarrow y)=m(x\rightsquigarrow y)=\max\{0,m(y)-m(x)\}$. 
Denote $\mathcal{M}_{BCK}(B)$ the set of all measures and $\mathcal{MM}_{BCK}(B)$ the set of all 
measure-morphisms on a pseudo BCK-algebra $B$. 

\begin{theo} \label{ms-eq-100} The following hold: \\
$(1)$ if $(A, \wedge, \thicksim, \backsim, 1)$ is a pseudo equality algebra, then 
$\mathcal{M}_{EQA}(A)\subseteq \mathcal{M}_{BCK}(\Psi(A));$ \\
$(2)$ if $(B,\wedge,\rightarrow, \rightsquigarrow,1)$ is a pseudo BCK(pC)-meet-semilattice, then 
$\mathcal{M}_{BCK}(B)\subseteq \mathcal{M}_{EQA}(\Phi(B))$. 
\end{theo}
\begin{proof}
The proof is straightforward.
\end{proof}

\begin{cor} \label{ms-eq-110} If $(A, \wedge, \thicksim, \backsim, 1)$ is an invariant pseudo equality algebra, then $\mathcal{M}_{EQA}(A)=\mathcal{M}_{BCK}(\Psi(A))$.
\end{cor}
\begin{proof}
Applying Theorems \ref{ps-eq-140}$(3)$, \ref{ms-eq-100} we have: \\
$\hspace*{2cm}$
$\mathcal{M}_{EQA}(A)\subseteq \mathcal{M}_{BCK}(\Psi(A))\subseteq \mathcal{M}_{EQA}(\Phi(\Psi(A)))=  
\mathcal{M}_{EQA}(A)$, \\
hence $\mathcal{M}_{EQA}(A)=\mathcal{M}_{BCK}(\Psi(A))$.
\end{proof}

\begin{theo} \label{ms-eq-120} The following hold: \\
$(1)$ if $(A, \wedge, \thicksim, \backsim, 1)$ is a pseudo equality algebra, then 
$\mathcal{MM}_{EQA}(A)\subseteq \mathcal{MM}_{BCK}(\Psi(A));$ \\
$(2)$ if $(B,\wedge,\rightarrow, \rightsquigarrow,1)$ is a linearly ordered pseudo BCK(pC)-meet-semilattice, then 
$\mathcal{MM}_{BCK}(B)\subseteq \mathcal{MM}_{EQA}(\Phi(B))$. 
\end{theo}
\begin{proof}
$(1)$ It is obvious. \\
$(2)$ Let $m\in \mathcal{MM}_{BCK}(B)$. If $x\le y$, thn $m(x)\ge m(y)$ and 
$m(x\wedge y\thicksim x)=m(x\backsim x\wedge y)=m(1)=0=max\{0,m(y)-m(x)\}$. If $y\le x$ we have: \\
$\hspace*{2cm}$
$m(x\wedge y\thicksim x)=m(y\thicksim x)=m(x\rightarrow y)=max\{0,m(y)-m(x)\}$ and \\
$\hspace*{2cm}$
$m(x\backsim x\wedge y)=m(x\backsim y)=m(x\rightsquigarrow x)=max\{0,m(y)-m(x)\}$, \\
hence $m\in \mathcal{MM}_{EQA}(\Phi(B))$. 
\end{proof}

\begin{cor} \label{ms-eq-130} If $(A, \wedge, \thicksim, \backsim, 1)$ is an invariant linearly ordered 
pseudo equality algebra, then $\mathcal{MM}_{EQA}(A)=\mathcal{MM}_{BCK}(\Psi(A))$.
\end{cor}
\begin{proof}
Applying Theorems \ref{ps-eq-140}$(3)$, \ref{ms-eq-120} we have: \\
$\hspace*{2cm}$
$\mathcal{MM}_{EQA}(A)\subseteq \mathcal{MM}_{BCK}(\Psi(A))\subseteq \mathcal{MM}_{EQA}(\Phi(\Psi(A)))=  
\mathcal{MM}_{EQA}(A)$, \\
hence $\mathcal{MM}_{EQA}(A)=\mathcal{MM}_{BCK}(\Psi(A))$.
\end{proof}

Internal states and state-morphism operators on pseudo equality algebras were defined and studied in \cite{Ciu8}. \\
Let $(A,\wedge,\thicksim,\backsim,1)$ be a pseudo equality algebra and 
$\sigma:A \longrightarrow A$ be a unary operator on $A$. For all $x, y\in A$ consider the following axioms: \\ 
$(IS_1)$ $\sigma (x)\le \sigma (y)$, whenever $x\le y;$ \\
$(IS_2)$ $\sigma(x\wedge y\thicksim x)=\sigma(y)\thicksim \sigma((x\wedge y\thicksim x) \backsim y)$ and 
         $\sigma(x\backsim x\wedge y)=\sigma(y\thicksim (x\backsim x\wedge y)) \backsim \sigma(y);$ \\
$(IS^{'}_2)$ $\sigma(x\wedge y\thicksim x)=\sigma(y)\thicksim \sigma((x\wedge y\thicksim y) \backsim x)$ and 
             $\sigma(x\backsim x\wedge y)=\sigma(x\thicksim (y\backsim x\wedge y)) \backsim \sigma(y);$ \\
$(IS_3)$ $\sigma(\sigma (x) \thicksim \sigma (y))= \sigma (x) \thicksim \sigma (y)$ and 
         $\sigma(\sigma (x) \backsim \sigma (y))= \sigma (x) \backsim \sigma (y);$ \\
$(IS_4)$ $\sigma(\sigma (x) \wedge \sigma (y))= \sigma (x) \wedge \sigma (y)$. \\
Then: \\
$(i)$ $\sigma$ is called an \emph{internal state of type I} or a \emph{state operator of type I} or 
a \emph{type I state operator} if it satisfies axioms $(IS_1)$, $(IS_2)$, $(IS_3), (IS_4);$ \\    
$(ii)$ $\sigma$ is called an \emph{internal state of type II} or a \emph{state operator of type II} or a 
\emph{type II state operator} if it satisfies axioms $(IS_1)$, $(IS^{'}_2)$, $(IS_3), (IS_4)$. \\
The structure $(A,\wedge,\thicksim,\backsim,\sigma,1)$ ($(A,\sigma)$, for short) is called a 
\emph{state pseudo equality algebra of type I (type II)}, respectively. 
Denote $\mathcal{IS}_{EQA}^{(I)}(A)$ and $\mathcal{IS}_{EQA}^{(II)}(A)$ the set of all internal states of 
type I and II on a pseudo equality algebra $A$, respectively. \\ 
Let ${\bf 1}_A, \Id_A:A\longrightarrow A$, defined by ${\bf 1}_A(x)=1$ and $\Id_A(x)=x$ for all $x\in A$. Then \\
${\bf 1}_A\in \mathcal{IS}_{EQA}^{(I)}(A), \mathcal{IS}_{EQA}^{(II)}(A)$ and $\Id_A \in \mathcal{IS}_{EQA}^{(I)}(A)$, 
but in general $\Id_A \notin \mathcal{IS}_{EQA}^{(II)}(A)$. \\
For $\sigma \in \mathcal{IS}_{EQA}^{(I)}(A)$ or $\sigma \in \mathcal{IS}_{EQA}^{(II)}(A)$, 
$\Ker(\sigma)=\{x\in A \mid \sigma(x)=1\}$ is called the \emph{kernel} of $\sigma$. 
An internal state of type I or type II is called \emph{faithful} if $\Ker(\sigma)=\{1\}$. 
In this case $(A,\sigma)$ is called a \emph{faithful state pseudo equality algebra of type I (type II)}, 
respectively. \\
A \emph{state-morphism operator} on a pseudo equality algebra $A$ is a map $\sigma:A\longrightarrow A$ satisfying the following conditions for all $x, y\in A$: \\
$(SM_1)$ $\sigma(x\thicksim y)=\sigma(x)\thicksim \sigma(y);$ \\
$(SM_2)$ $\sigma(x\backsim y)=\sigma(x)\backsim \sigma(y);$ \\
$(SM_3)$ $\sigma(x\wedge y)=\sigma(x)\wedge \sigma(y);$ \\
$(SM_4)$ $\sigma(\sigma(x))=\sigma(x)$. \\
The pair $(A, \sigma)$ is called a \emph{state-morphism pseudo equality algebra}.

\begin{prop} \label{is-eq-10} $\rm($\cite{Ciu8}$\rm)$ If $(A,\sigma)$ is a state pseudo equality algebra of 
type I or type II, then the following hold:\\
$(1)$ $\sigma(1)=1;$ \\
$(2)$ $\sigma(\sigma(x))=\sigma(x)$, for all $x\in A;$ \\
$(3)$ $\sigma(A)=\{x\in A \mid x=\sigma(x)\};$ \\ 
$(4)$ $\sigma(A)$ is a subalgebra of $A;$ \\
$(5)$ $\Ker(\sigma)\cap \Img(\sigma)=\{1\};$ \\
$(6)$ $\Ker(\sigma)\in {\mathcal DS}(A);$ \\
$(7)$ if $A$ is invariant, then $\Ker(\sigma)$ is a subalgebra of $A$.
\end{prop}

\begin{prop} \label{is-eq-20} Let $(A,\sigma)$ be a state pseudo equality algebra. 
Then $\mathcal{IS}_{EQA}^{(I)}(A)=\mathcal{IS}_{EQA}^{(II)}(A)$ if and only if $A$ is a commutative pseudo 
equality algebra. 
\end{prop}
\begin{proof} 
Since $A$ is commutative, then axioms $(IS_2)$ and $(IS^{'}_2)$ coincide, so    
$\mathcal{IS}_{EQA}^{(I)}(A)=\mathcal{IS}_{EQA}^{(II)}(A)$. 
Conversely, if $\mathcal{IS}_{EQA}^{(I)}(A)=\mathcal{IS}_{EQA}^{(II)}(A)$, since 
$\Id_A\in \mathcal{IS}_{EQA}^{(I)}(A)$, it follows that $\Id_A\in \mathcal{IS}_{EQA}^{(II)}(A)$.   
From axiom $(IS^{'}_2)$ we get 
$x\wedge y\thicksim x=y\thicksim ((x\wedge y\thicksim y) \backsim x)$ and 
$x\backsim x\wedge y=(x\thicksim (y\backsim x\wedge y)) \backsim y$.  
Hence by Theorem \ref{com-eq-70}, $A$ is commutative. 
\end{proof}

\begin{prop} \label{is-eq-30} Let $(A,\sigma)$ be a faithful commutative state pseudo equality algebra 
of type I or type II and $x, y\in A$. Then: \\
$(1)$ $x< y$ implies $\sigma(x)< \sigma(y);$ \\
$(2)$ if $A$ is linearly ordered, then $\sigma=Id_A$.  
\end{prop}
\begin{proof}
$(1)$ Consider $\sigma \in \mathcal{IS}_{EQA}^{(I)}(A)$ and $x, y\in A$ such that $x< y$, 
that is $\sigma(x)\le \sigma(y)$. \\
Suppose that $\sigma(x)=\sigma(y)$. Applying $(IS_2)$ and the commutativity of $A$ we get: \\
$\hspace*{2cm}$
$\sigma(x\thicksim y)=\sigma(x\wedge y\thicksim y)=\sigma(x)\thicksim \sigma((x\wedge y\thicksim y)\backsim x)$ \\
$\hspace*{3.6cm}$
$=\sigma(x)\thicksim \sigma((x\wedge y\thicksim x)\backsim y)=\sigma(x)\thicksim \sigma(y)=1$ and \\
$\hspace*{2cm}$
$\sigma(x\backsim y)=\sigma(x\backsim x\wedge y)=\sigma(y\thicksim (x\backsim x\wedge y))\backsim \sigma(y)= 
\sigma(y)\backsim \sigma(y)=1$. \\
It follows that $x\thicksim y, x\backsim y\in \Ker(\sigma)$, that is $x\thicksim y=x\backsim y=1$. 
According to Proposition \ref{ps-eq-10}$(3)$, we get $x=y$, a contradiction. Hence $\sigma(x)<\sigma(y)$. 
Similarly for $\sigma \in \mathcal{IS}_{EQA}^{(II)}(A)$. \\
$(2)$ Assume that $A$ is linearly ordered and let $x\in A$ such that $\sigma(x)\ne x$, that is 
$\sigma(x)<x$ or $x<\sigma(x)$. Applying $(1)$ we get $\sigma(\sigma(x))<\sigma(x)$ or 
$\sigma(x)<\sigma(\sigma(x))$. This a contradiction, so $\sigma(x)=x$ for all $x\in A$, that is $\sigma=Id_A$.   
\end{proof}

\begin{theo} \label{is-eq-40} Let $(A,\sigma)$ be a commutative symmetric linearly ordered state pseudo equality 
algebra. Then $\mathcal{IS}_{EQA}^{(I)}(A)=\mathcal{IS}_{EQA}^{(II)}(A)=\mathcal{SM}_{EQA}(A)$.
\end{theo}
\begin{proof}
According to \cite[Rem. 7.9]{Ciu8}, 
$\mathcal{SM}_{EQA}(A)\subseteq \mathcal{IS}_{EQA}^{(I)}(A)=\mathcal{IS}_{EQA}^{(II)}(A)$. \\
Conversely, by Proposition \ref{is-eq-20}, $\mathcal{IS}_{EQA}^{(I)}(A)=\mathcal{IS}_{EQA}^{(II)}(A)$. \\
Consider $\sigma \in \mathcal{IS}_{EQA}^{(I)}(A)$ and $x, y\in A$ such that $x\le y$. We have: \\
$\hspace*{2cm}$
$\sigma(x\thicksim y)=\sigma(x\wedge y\thicksim y)=\sigma(x)\thicksim \sigma((x\wedge y\thicksim y)\backsim x)$ \\
$\hspace*{3.6cm}$
$=\sigma(x)\thicksim \sigma((x\wedge y\thicksim x)\backsim y)=\sigma(x)\thicksim \sigma(y)$ and  \\
$\hspace*{2cm}$
$\sigma(x\backsim y)=\sigma(y\thicksim x)=\sigma(y)\thicksim \sigma(x)=\sigma(x)\backsim \sigma(y)$. \\
Since by $(IS_1)$, $\sigma(x)\le \sigma(y)$, we get $\sigma(x\wedge y)=\sigma(x)=\sigma(x)\wedge \sigma(y)$. \\
The case $y\le x$ can be treated similarly, thus $\sigma$ satisfies axioms $(SM_1)$, $(SM_2)$,$(SM_3)$. \\
Since axiom $(SM_4)$ follows by Proposion \ref{is-eq-10}$(2)$, we proved that $\sigma\in \mathcal{SM}_{EQA}(A)$. 
Hence $\mathcal{SM}_{EQA}(A)\subseteq \mathcal{IS}_{EQA}^{(I)}(A)=\mathcal{IS}_{EQA}^{(II)}(A)$ and we conclude that  $\mathcal{IS}_{EQA}^{(I)}(A)=\mathcal{IS}_{EQA}^{(II)}(A)=\mathcal{SM}_{EQA}(A)$.
\end{proof}

\begin{cor} \label{is-eq-50} If $(A,\sigma)$ is a finite invariant commutative linearly ordered state pseudo equality 
algebra, then $\mathcal{IS}_{EQA}^{(I)}(A)=\mathcal{IS}_{EQA}^{(II)}(A)=\mathcal{SM}_{EQA}(A)$.
\end{cor}
\begin{proof}
It follows by Proposition \ref{com-eq-50} and Theorem \ref{is-eq-40}. 
\end{proof}

$\vspace*{5mm}$

\section{Valuations on pseudo equality algebras} 

In this section the notions of pseudo-valuation and commutative pseudo-valuation on pseudo equality algebras are defined and investigated. Given a pseudo equality algebra $A$, it is proved that the kernel of a commutative pseudo-valuation on $A$ is a commutative deductive system of $A$. If moreover $A$ is commutative, then we prove 
that any pseudo-valuation on $A$ is commutative. \\
In what follows by $A$ we will denote a pseudo equality algebra.  

\begin{Def} \label{va-eq-10} A real-valued function $\varphi:A\longrightarrow {\mathbb R}$ is called a 
\emph{pseudo-valuation} on $A$ if it satisfies the following conditions: \\
$(pv_1)$ $\varphi(1)=0;$ \\
$(pv_2)$ $\varphi(y)-\varphi(x)\le \min\{\varphi(x\wedge y\thicksim x), \varphi(x\backsim x\wedge y)\}$  
for all $x, y\in A$. \\
A pseudo-valuation $\varphi$ is said to be a \emph{valuation} if it satisfies the condition: \\
$(pv_3)$ $v(x)=0$ implies $x=1$ for all $x\in A$. 
\end{Def}

Denote $\mathcal{PV}_{EQA}(A)$ the set of all pseudo-valuations on $A$. 

\begin{prop} \label{va-eq-20} If $\varphi\in \mathcal{PV}_{EQA}(A)$, then the following hold for all $x, y, z\in A$: \\
$(1)$ $\varphi(x)\ge \varphi(y)$, whenever $x\le y$ ($\varphi$ is order reversing); \\
$(2)$ $\varphi(x)\ge 0;$ \\
$(3)$ if $(y\backsim x\wedge y)\wedge z\thicksim z=1$ or $z\backsim (x\wedge y\thicksim y)\wedge z=1$, then 
      $\varphi(x)\le \varphi(y)+ \varphi(z)$. 
\end{prop}
\begin{proof}
For $x\le y$ we have $\varphi(y)-\varphi(x)\le \min\{\varphi(x\wedge y\thicksim x), \varphi(x\backsim x\wedge y)\}= 
\min \{\varphi(1),\varphi(1)\}=0$, that is $\varphi(x)\ge \varphi(y)$. \\
$(2)$ Since $x\le 1$, by $(1)$ we get $\varphi(x)\ge \varphi(1)=0$. \\
$(3)$ By $(pv_2)$ we have 
$\varphi(y\backsim x\wedge y)-\varphi(z)\le \varphi((y\backsim x\wedge y)\wedge z\thicksim z)=\varphi(1)=0$, so 
$\varphi(y\backsim x\wedge y)\le \varphi(z)$. Applying again $(pv_2)$ we get 
$\varphi(x)-\varphi(y)\le \varphi(y\backsim x\wedge y)\le \varphi(z)$, that is $\varphi(x)\le \varphi(y)+ \varphi(z)$. 
Similarly for the case $z\backsim (x\wedge y\thicksim y)\wedge z=1$. 
\end{proof}

\begin{ex} \label{va-eq-20-10} (see \cite{Bus3})
Let $D\in \mathcal{DS}(A)$ and the map $\varphi:A\longrightarrow {\mathbb R}$ defined by 
\begin{equation*}
\varphi(x)= \left\{
\begin{array}{cc}
	0,\:\: {\rm if} \: x \in D\\
	a,\:\: {\rm if} \: x \notin D,
\end{array}
\right.
\end{equation*}
where $a\in {\mathbb R}$, $a\ge 0$. Then $\varphi\in \mathcal{PV}_{EQA}(A)$. 
\end{ex}

For $\varphi\in \mathcal{PV}_{EQA}(A)$, denote $D_{\varphi}=\{x\in a\mid \varphi(x)=0\}$, called the 
\emph{kernel} of $\varphi$. 

\begin{prop} \label{va-eq-30} $D_{\varphi}\in \mathcal{DS}(A)$, for every $\varphi\in \mathcal{PV}_{EQA}(A)$. 
\end{prop}
\begin{proof}
Since by $(pv_1)$, $\varphi(1)=0$, it follows that $1\in D_{\varphi}$, that is $(DS_1)$ is satisfied. \\
Let $x, y \in A$ such that $x\le y$ and $x\in D_{\varphi}$. Then $0=\varphi(x)\ge \varphi(y)$, hence 
$\varphi(y)=0$, that is $y\in D_{\varphi}$. Thus $(DS_2)$ is also verified. 
Consider $x, y\in A$ such that $x, x\wedge y\thicksim x\in D_{\varphi}$. It follows that 
$\varphi(y)-\varphi(x)\le \varphi(x\wedge y\thicksim x)$, so $\varphi(y)-0\le 0$, hence $\varphi(y)=0$. 
Thus $y\in D_{\varphi}$ and $(DS_4)$ is satisfied. We conclude that $D_{\varphi}\in \mathcal{DS}_{EQA}(A)$. 
\end{proof}

\begin{Def} \label{va-eq-40} A pseudo-valuation $\varphi$ on $A$ is said to be \emph{commutative} if it satisfies the following conditions for all $x, y\in A$: \\
$(cpv_1)$ $\varphi(x\thicksim x\vee_1 y)\le \varphi(x\wedge y\thicksim y),$ \\
$(cpv_2)$ $\varphi(x\vee_2 y\backsim x)\le \varphi(y\backsim x\wedge y)$. 
\end{Def}

Denote $\mathcal{PV}^{c}_{EQA}(A)$ the set of all commutative pseudo-valuations on $A$. 

\begin{prop} \label{va-eq-50} A pseudo-valuation $\varphi$ on $A$ is commutative if and only if it satisfies the following conditions for all $x, y, z\in A:$ \\
$(cpv_3)$ $\varphi(x\thicksim x\vee_1 y)\le \varphi(z\backsim (x\wedge y\thicksim y)\wedge z)+\varphi(z),$ \\
$(cpv_4)$ $\varphi(x\vee_2 y\backsim x)\le \varphi((y\backsim x\wedge y)\wedge z\thicksim z)+\varphi(z)$.
\end{prop}
\begin{proof}
Let $\varphi$ be a commutative pseudo-valuation on $A$, that is $\varphi$ satisfies conditions $(cpv_1)$ and $(cpv_2)$. 
By $(cpv_1)$ and $(pv_2)$ we have: $\varphi(x\thicksim x\vee_1 y)-\varphi(z)\le 
\varphi(x\wedge y\thicksim y)-\varphi(z)\le \varphi(z\backsim (x\wedge y\thicksim y)\wedge z)$, that is $(cpv_3)$. 
Similarly from $(cpv_2)$ and $(pv_2)$ we get $(cpv_3)$. \\
Conversely, let $\varphi$ be a pseudo-valuation on $A$ satisfying conditions $(cpv_3)$ and $(cpv_4)$. 
Taking $z=1$ we get $(cpv_1)$ and $(cpv_2)$, hence $\varphi$ is commutative.
\end{proof}

\begin{prop} \label{va-eq-60} If $\varphi\in \mathcal{PV}^{c}_{EQA}(A)$, then $D_{\varphi}\in \mathcal{DS}_c(A)$.  
\end{prop}
\begin{proof}
Let $x, y\in A$ such that $x\wedge y\thicksim y\in D_{\varphi}$, that is $\varphi(x\wedge y\thicksim y)=0$.  
By $(cpv_1)$, $\varphi(x\thicksim x\vee_1 y)\le \varphi(x\wedge y\thicksim y)=0$, hence  
$\varphi(x\thicksim x\vee_1 y)=0$, so $x\thicksim x\vee_1 y\in D_{\varphi}$. 
Similary from $y\backsim x\wedge y\in D_{\varphi}$, applying $(cpv_2)$ we get $x\vee_2 y\backsim x\in D_{\varphi}$. 
Thus $D_{\varphi}\in \mathcal{DS}_c(A)$.
\end{proof}

\begin{prop} \label{va-eq-70} If $A$ is commutative, then $\mathcal{PV}^{c}_{EQA}(A)=\mathcal{PV}_{EQA}(A)$.
\end{prop}
\begin{proof} Let $\varphi\in \mathcal{PV}_{EQA}(A)$ and $x, y\in A$. 
Applying Proposition \ref{ps-eq-30} we have  
$x\wedge y\thicksim y=x\thicksim ((x\wedge y\thicksim y)\backsim x)=x\thicksim y\vee_1 x=x\thicksim x\vee_1 y$, hence 
$\varphi(x\wedge y\thicksim y)=\varphi(x\thicksim x\vee_1 y)$.
Similarly $\varphi(y\backsim x\wedge y)=\varphi(x\vee_2 y\backsim x)$, hence $(cpv_1)$ and $(cpv_2)$ are satisfied. 
Thus $\varphi\in \mathcal{PV}^{c}_{EQA}(A)$. We conclude that $\mathcal{PV}^{c}_{EQA}(A)=\mathcal{PV}_{EQA}(A)$.
\end{proof}

\begin{rem} \label{va-eq-80} 
Given a pseudo BCK-meet-semilattice $(B,\wedge, \rightarrow, \rightsquigarrow,1)$, a pseudo-valuation on $B$ 
can be defined as a real-valued function $\varphi:B\longrightarrow {\mathbb R}$ satisfying the conditions: 
$(i)$ $\varphi(1)=0$, $(ii)$ $\varphi(y)-\varphi(x)\le \min\{\varphi(x\rightarrow y), \varphi(x\rightsquigarrow y)\}$  
for all $x, y\in B$. Denote $\mathcal{PV}_{BCK}(B)$ the set of all pseudo-valuations on $B$. 
Similarly as in the case of measures and internal states we can show that
$\mathcal{PV}_{EQA}(A)\subseteq \mathcal{PV}_{BCK}(\Psi(A))$ and 
$\mathcal{PV}_{BCK}(B)\subseteq \mathcal{PV}_{EQA}(\Phi(B))$. 
\end{rem}

$\vspace*{5mm}$

\section{Concluding remarks}

As mentioned in the Introduction, a new concept of FTT has been developed with the structure of truth values 
formed by a linearly ordered good EQ$_{\Delta}$-algebra (\cite{Nov9}) and a fuzzy-equality based logic 
called EQ-logic has also been introduced (\cite{Nov10}).
The study of pseudo equality algebras has the purpose to develop appropriate algebraic semantics 
for FTT, so a concept of FTT should be introduced based on these algebras.  
At the same time, pseudo equality algebras could be intensively studied from an algebraic point of view.
Commutativity property proved to play an important role for studying states, measures and internal states on 
multiple-valued logic algebras.
In this paper we defined and studied the commutative pseudo equality algebras and the commutative deductive 
systems of pseudo equality algebras. We proved certain results regarding the commutative deductive systems of 
pseudo equality algebras and we gave a characterization of commutative pseudo equality algebras in terms of 
commutative deductive systems. We applied these results to investigate the measures, measure-morphisms,  
internal states and pseudo-valuations on pseudo equality algebras.
As another direction of research, one could find axiom systems for commutative pseudo equality algebras and 
prove similar results as in \cite{Ciu7} for commutative pseudo BCK-algebras.

$\vspace*{5mm}$

\section* {\bf\leftline {Compliance with Ethical Standards}}

Confict of interest: The author declares that he has no conflict of interest.




$\vspace*{5mm}$

\setlength{\parindent}{0pt}

\vspace*{3mm}

\begin{flushright}
\begin{minipage}{148mm}\sc\footnotesize
Lavinia Corina Ciungu\\
Department of Mathematics \\
University of Iowa \\
14 MacLean Hall, Iowa City, Iowa 52242-1419, USA \\
{\it E--mail address}: {\tt lavinia-ciungu@uiowa.edu}

\end{minipage}
\end{flushright}

\end{document}